\documentclass[11pt,letterpaper]{amsart}
\usepackage{amsmath,amsfonts,amssymb,amscd,amsthm,amsbsy,epsf}
\usepackage[dvips]{graphicx}
\usepackage{comment}
\usepackage[dvips]{color}
\usepackage{xcolor}

\numberwithin{equation}{section}
\newtheorem{theorem}{Theorem}
\newtheorem{meta-thm}[theorem]{Meta-Theorem}
\newtheorem{lemma}[theorem]{Lemma}

\newtheorem{proposition}[theorem]{Proposition}

\newtheorem{definition}[theorem]{Definition}
\theoremstyle{remark}
\newtheorem{remark}[theorem]{Remark}

\def\Id{\operatorname{Id}}

\def\E{{\mathcal E}}

\def\LL{ {\mathcal{L}}}

\def\complex{{\mathbb C}}
\def\integer{{\mathbb Z}}
\def\nat{{\mathbb N}}

\def\real{{\mathbb R}}

\def\torus{{\mathbb T}}

\def\eps{\varepsilon}
\def\teps{ {\tilde \eps}}
\def\th{\theta}

\def\om{\omega}

\def\<{\langle}
\def\>{\rangle}
\def\froeschle{Froeschl\'e \  }

\begin{document}
\title[Simple proof of Gevery Estimates]{A simple proof of Gevrey estimates for expansions of quasi-periodic orbits: dissipative models and lower dimensional tori}

\author[A.P. Bustamante]{Adri\'an P. Bustamante}  
\address{
Department of Mathematics, University of Roma Tor Vergata, Via
della Ricerca Scientifica 1, 00133 Roma (Italy)}
\email{bustamante@mat.uniroma2.it}
\thanks{A.B. acknowledges the MIUR Excellence Department Project
awarded to the Department of Mathematics, University of Rome Tor
Vergata, CUP E83C18000100006.  A.B. was partially supported by the MIUR-PRIN 20178CJA2B “New Frontiers of Celestial Mechanics: theory and Applications”} 

\author[R. de la Llave]{Rafael de la Llave} 
\address{
School of Mathematics,
Georgia Institute of Technology,
686 Cherry St. Atlanta GA. 30332-1160, USA }
\email{rafael.delallave@math.gatech.edu}



\begin{abstract}
  We consider standard-like/\froeschle dissipative maps
  with a dissipation and  nonlinear perturbation.
  That is
  \[
    T_\eps(p,q) = \left(
      (1 - \gamma \eps^3) p + \mu  +  \eps V'(q),
      q +  (1 - \gamma \eps^3) p + \mu  +  \eps V'(q)    \mod 2 \pi \right)
      \]
      where $p \in \real^D$, $q \in 2 \pi \torus^D$ are the dynamical
      variables.  The $\mu \in \real^D, \gamma\in \real$ are parameters
      of the model. We will consider the dependence on $\eps \in \real$,
      which is the most important perturbation parameter. We assume that the potential $V$ is a trigonometric polynomial
      (we conjecture, however, that this is not needed for the results).

      Note that when $\gamma \ne 0$, the perturbation parameter $\eps$ creates
      dissipation, which has a drastic effect on the existence of
      quasi-periodic orbits, hence it is a singular perturbation.

      We fix a frequency $\omega \in \real^D$ and study the existence of
      quasiperiodic orbits. When there is dissipation, having
      a quasiperiodic orbit  of frequency $\omega$ requires
      adjusting the parameter $\mu$, called \textit{the drift}. 

      We first study the Lindstedt series (formal power series in $\eps$)  for quasiperidic orbits with
      $D$ independent frequencies and the drift when $\gamma \ne 0$.
      We show that, when $\omega$ is
      irrational, the series exist to all orders, and when  $\omega$ is Diophantine, 
we show that the formal Lindstedt  series are Gevrey
(we do not consider the case $\gamma = 0$, since
\cite{Moser67} showed that, for Diophantine $\omega$,
the series have a positive radius of convergence).

  The Gevrey nature of the Lindstedt series above is a  particular case of the results of a generally
  applicable method in \cite{BustamanteL22}, but the present proof is
  rather elementary. 

  We  also study the case when $D = 2$, but the quasi-periodic orbits
  have only a one independent frequency (Lower dimensional tori).
  Both when $\gamma = 0$ and when $\gamma \ne 0$, we show
  that, under some mild non-degeneracy conditions on $V$, there
  are (at least two) formal Lindstedt  series defined to all orders
  and that they are Gevrey. Furthermore, we can take $\mu$ along a
  $1$ dimensional space.

 \end{abstract} 
\maketitle


\section{Introduction}\label{sec:intro}
The goal of this paper is to show that several perturbative
formal power series expansions (See Definition~\ref{def:formalsolutions})
for
quasi-periodic solutions (tori)  are Gevrey, that is, there
are specific upper bounds on the growth of the coefficients of the power series, see Definition~\ref{gevrey}.

The above formal series
for maximal dimensional tori in symplectic mappings,
were shown to have a  positive radius of
convergence in \cite{Moser67}. Hence, in this paper we
will consider the series for  maximal dimensional tori
only when the perturbations include weak  dissipation.

Adding dissipation is a singular perturbation since the tori may be
isolated and we need to include the drift parameter to ensure
their existence. When there is dissipation, the
drift that ensures the existence
of a torus with the frequency $\omega$ needs to be determined for
each value of the perturbation, so that the drift becomes an unknown. 

We will also consider series for
lower dimensional tori, both in the conservative case
(studied in \cite{Jo-Lla-Zou-99} ) and
in the disipative case. In both cases,
it is expected that the series have zero radius of convergence, see \cite{Jo-Lla-Zou-99, GallavottiG02}.
We show that, when the frequency is Diophantine,  the perturbative formal power series are Gevrey. 

The paper \cite{BustamanteL22} developed a general method, applicable 
to many problems, 
to prove that formal power series in the pertubation parameter
of quasi-periodic solutions
of Hamiltonian systems perturbed by dissipation
are Gevrey   ( See Definition~\ref{gevrey}). The method of \cite{BustamanteL22} is based on
implementing a fast converging method (KAM type)
for expansions  that converges superexponentially
but in exponentially fast decreasing  balls.  Even if the only place where
all the iterative steps are defined is the origin, by using Cauchy estimates
after  a finite number of steps, one obtains Gevrey estimates on
the coefficients. This method requires that the frequency satisfies
Diophantine estimates.

The example used in \cite{BustamanteL22} to illustrate the results is  perturbative
expansions of quasiperiodic orbits with Diophantive frequency 
of slighty dissipative  twist maps (see Section~\ref{sec:model}). 
The unperturbed case is the integrable twist map, but the perturbation
in Section ~\ref{sec:model} involves
not only a nonlinearity but also dissipation. Including 
dissipation requires to introduce an extra parameter -- often
called \emph{the drift}, that has the interpretation of an external forcing.
The value of the drift that allows to have the quasi-periodic solution of the selected frequency is
part of the unknowns that need to be determined.

The first  goal of this paper is to give an elementary proof of
a particular case of the results in \cite{BustamanteL22}.
We consider a simple model, the dissipative standard map
with any number of degrees of freedom, see Section~\ref{sec:model},
which has been studied in the literature (e.g. \cite{Chirikov,Froeschle71}) and indeed,
was the model used as an example  in \cite{CallejaCL22}
and studied numerically in \cite{BustamanteC19, BustamanteC21}. 
We  show that the series are Gevrey.

The second goal of this paper is to extend the methods to
lower dimensional tori (1-dimensional tori in maps
with 2 degrees of freedom). We consider some  twist maps with two
degrees of freedom (also called \froeschle maps) and
seek solutions with only one independent frequency. 

The 1-D tori in 2 DOF systems involve
extra considerations related to the different topological
ways to embed a 1-D torus in the phase space $\torus^2 \times \real^2$
and we recall them.

In the conservative case the formal power series  were studied in
\cite{Jo-Lla-Zou-99}, where it was showed  that, under mild non-degeneracy
conditions in $V$, the series exist to all orders
\footnote{
The goal of \cite{Jo-Lla-Zou-99} was to study the domain of
definition in the complex $\eps$
plane of the lower dimensional tori as well as properties of the linearization
around them. Similar results on the domain were proved in
\cite{GallavottiG02} by resummation of series.

One can find formal expansions of lower dimensional tori in
Hamiltonian flows 
in the classical literature \cite[Chapter V]{Poincare}, \cite{PoincareS}, but they are not
easy to read by modern readers. See also \cite{CorsiG08} }.

In this paper we show that the series constructed in
\cite{Jo-Lla-Zou-99} are Gevrey when the frequency is Diophantine. 

We also study the lower dimensional tori
in the case that the perturbation includes disipation
and show that the perturbative expansions for the parameterization
of the torus and the drift are Gevrey when the frequency is Diophantine.

The problem of lower dimensional tori is interesting for the theory of
perturbative expansions 
because the small divisors involved in these  formal power series are less
than those needed for a KAM treatment. 
Whiskered KAM tori are also useful in Arnold diffusion and in mission design
and the theoretical insight leads to practical algorithms, which have been implemented.

We present results for maps. The method can be adapted for flows using
the algorithms for differential equations. The method of
taking a time-one map does not seem to work since the time one map of hamiltonians with a trigonometric  perturbation are not trigonometric perturbations.

\section{Description of the model considered and of the 
  perturbative expansions of quasi-periodic orbits}
\label{sec:preliminaries}
In this section,  we describe the models that we 
study and collect standard definitions such as the notion of quasiperiodic orbits and the expansions on quasi-periodic
motions.   All the material in this section is quite standard, but we will need it to set the notation used
in the statement of the results (Section~\ref{sec:statement}) and
the proofs. 

\subsection{Description of the models considered} 
\label{sec:model}

We will consider maps
that transform a point  $(p_n, q_n) \in \real^D \times \torus^D$ 
into another point $(p_{n+1}, q_{n+1}) \in \real^D \times \torus^D$
according to the formula: 

\begin{equation}\label{DSM}
  \begin{split}
 &   p_{n+1} = (1 - \gamma \eps^3) p_n + \mu  +  \eps  V'(q_n ) \\
 &  q_{n+1}  = q_n +   p_{n+1}  \mod 2 \pi
\end{split}
\end{equation}
where $V$  is a real valued periodic function of its argument.  $V'$ denotes the gradient 
with respect to the arguments.   The notation $V'$ is
for the gradient, hence, it is a $D$ dimensional vector. Maps of this form have been used as qualitative models of dynamics near resonances, \cite{Chirikov}, \cite{Froeschle71}.
In the case that $D=1$, the maps are called \emph{standard-like} maps
(or Chirikov maps) when $V(x)= \sin(x)$. When $D= 2$, they are called
\froeschle maps.

The dimension $D$ does not affect the algebraic manipulations described in
this section, but it affects some of the geometric considerations
for lower dimensional tori discussed in Section ~\ref{sec:recursive-lower}. 

The formula \eqref{DSM}  describing our maps involves several parameters. The most important 
one in this study is $\eps$.
Note that for $\eps = 0$, $\mu =0$ -- the unperturbed integrable case -- the map \eqref{DSM} has a very simple structure,
the orbits are periodic or quasi-periodic depending on $p$.
Since $p$ is conserved, the tori of fixed $p$ are invariant, the
motion on them is a rotation. The $D$ dimensional torus is
foliated by invariant tori of dimension $L$, where $0 <  L \le D$
is the number of rationaly independent components of $p$.
When $L = D$ we talk about fully dimensional tori. 

The effect of $\eps$ is to introduce a non-linearity (in the terms
involving $V'$). When $\gamma = 0$  the maps remain symplectic,
but when $\gamma \ne 0$ the perturbation includes a
(weak) dissipation. Adding a dissipation affects very drastically the existence
of quasi-periodic orbits, so it is a sigular perturbation.
When $\gamma \ne 0$ we need to introduce an extra parameter
$\mu$ to ensure the existence of orbits with the prescribed frequency
$\omega$. The parameter $\mu$ is  an unknown  to be determined.

In this paper we formulate an invariance equation
\eqref{periodicity-hull} for a quasiperiodic solution
( both for  maximal dimensional and for lower dimensional tori).
In the case of Hamiltonian vector fields it is well known
that the formal power series exist to all orders.
For the sake of completeness, we also present a proof
for maps  that involve dissipation and show  that 
equation~\eqref{periodicity-hull} can be solved in  the sense of formal
power series (See Definition~\ref{def:formalsolutions}).

The main result of this paper is to show that these formal power
series satisfy Gevrey estimates
( See Definition~\ref{gevrey} )  
 when the frequency is Diophantine.

\begin{remark}
  The formal power series  solutions have the property that 
  that sums up to order $N$
  produce  residuals  that are smaller (in some appropriate norm)
  than $C_N |\eps|^{N+1}$. Having an approximate  solution of \eqref{periodicity-hull}
  with small residual is the main hypothesis of \emph{a-posteriori}
  theorems that establish that, given the approximate solution (and some
  non-degeneracy conditions) there is a true solution nearby. 
  See   \cite{CallejaCL13} for maximal dimensional tori
  and \cite{CallejaCL20} for lower dimensional tori. 

  Using the approximate solution produced here by truncating the
  formal power series, followed by an a-posteriori theorem, 
  gives lower  estimates of the domain of $\eps$ in the complex plane
  for which there are quasi-periodic solutions.  See \cite{CallejaCL17}
  for the case of maximal dimensional tori and \cite{Jo-Lla-Zou-99} for
  the case of lower dimensional tori. Related  results about the domain of $\eps$ were
  proved using resummation methods in the series  in \cite{GallavottiG02}.

The results proved  here on the Gevrey nature of the series give quantitative
information on the error of the invariance equation for the sum of the
first $N$ terms of the formal power series and allow to make more
quantitative the results based on a-posteriori methods, \cite{Ca-Ce-Lla-13, CallejaCL20}
\end{remark}

\begin{remark}
  The a-posteriori results mentioned before are based on
  showing the convergence of an iterative procedure.  This iterative
  procedure can be implemented in the computer as an
  extremely efficient algorithm (low storage requirements, low
  operation count, quadratically convergent) which is also very
  well adapted to being programed in modern languages.

  One can take truncations of the series here as a starting
  point of the iterative methods \cite{CallejaCL13, CallejaCL20}.  Some
  implementations for the computations of whiskered tori in
  problems in celestial mechanics are in \cite{KumarAL22, FernandezHM22}.
  \end{remark}

In this paper we will assume that $V$, and hence $V'$,  is a trigonometric
polynomial, we use the notation
\begin{equation} \label{polynomial}
  V' (\th) =  \sum_{|\ell|\leq J, \;\ell\in \integer^D} \alpha_\ell e^{ i \ell\cdot  \th}. 
\end{equation}
Note that, with the notation in \eqref{polynomial}, $\alpha_\ell\in \real^D$
and is proportional to $\ell$. We also have $\alpha_0 = 0$ and for real valued polynomials one has $\alpha_\ell = \alpha_{-\ell}^*$.

\begin{remark}
We conjecture that the assumption that $V$ is a trigonometric
  polynomial can be weakened to $V$ analytic.
  It seems plausible that this can be obtained by extending the
  methods of \cite{BustamanteL22}.
\end{remark}

\subsection{A second order equation formulation}
It is convenient to transform the equation \eqref{DSM} 
for orbits into a second order equation for the orbits in half of the variables. 

If we use the equation in \eqref{DSM}, we
obtain $p_{n+1} = q_{n+1} - q_n$ and if we substitute this in
the first equation in \eqref{DSM}, we obtain
\begin{equation}\label{lagrangian}
  q_{n+1} + (1 - \gamma \eps^3) q_{n-1} - (2 - \gamma \eps^3) q_n - \eps V'(q_n) - \mu = 0
\end{equation}

The second order $D$ dimensional equation \eqref{lagrangian} is equivalent to the 
first order   system
\eqref{DSM} in $2D$ variables. The second
order formulation \eqref{lagrangian}  will be more convenient for our purposes.

\begin{remark}
  The case $\gamma = 0$ is the case of conservative perturbations.
  In this case we can take $\mu = 0$.

  The study of Lindstedt series of maximal dimensional
  tori  
when $\gamma  = 0$ is classic.
  The paper \cite{Moser67} showed that, when the frequency is Diophantine,
  the Lindstedt series have a positive radius of convergence.

  In the case when $\gamma > 0$, as shown in \cite{CallejaCL17}, the
  radius of convergence is zero, no matter how small is $\gamma$.
  This is an indication of the fact that adding dissipation to a conservative
  system is a very singular perturbation for the long term behavior.

  In this paper we will concentrate on the case that $\gamma > 0$
  for maximal dimensional tori.
\end{remark}

\begin{remark}
  A simple calculation done in \cite{Jo-Lla-Zou-99} shows
  that, under some mild non-degeneracy assumption,
  the Lindstedt series predicts that the lower dimensional tori
  have hyperbolic normal directions for $\eps > 0$
  and elliptic for $\eps < 0$ (or viceversa).

  In the normally elliptic case, resonances between the normal
  direction and the internal direction can lead to lack of
  analyticity of the torus.  This shows that, for general potentials
  one may get that the radius of convergence of the Lindstedt series
  for a ``generic'' $V$ is
  zero for any $\gamma$. 

  Even if the argument, as written does not rigorously apply
  to forcings which are trigonometric polynomial, it suggests that
  even for trigonometric polynomial forcings, we should expect that
  the radius of convergence is zero.
  \end{remark}

  \subsection{Quasiperiodic orbits: hull functions and their
    periodicity} 
\label{sec:quasiperiodic}

We recall that given $\omega\in\real^L$
we say that a sequence
$q_n$ is quasiperiodic of frequency $\omega$ if
we can find a function $h$ such that
\begin{equation} \label{hulldefined} 
  q_n = h(\omega n )
\end{equation}
with $h$ -- often called
the hull function where $h$ is a function defined on $2 \pi \torus^L$
and taking values in $2 \pi \torus^D$.  $L$ is the number of independent frequencies in $\omega$ --
and $\torus^L$ is the space of internal phases.
We can always assume that
$\omega \cdot l \ne 0 \;\forall l \in \integer^L  \setminus \{0\} $.
Otherwise, we could find an equivalent description of $q$ with a
hull function in a torus of less dimension.

Note that we can think of $h$ as an embedding of the torus  $\torus^L$ into
the phase space. Hence, it is common to refer to quasi-periodic
orbits as invariant tori. 

\begin{remark}
In the literature there are some papers which reserve
the name quasi-periodic for the case that the functions $h$ are differentiable,
and use other names such as almost automorphic for lower regularity. Other
papers use quasi-periodic irrespective of the regularity of $h$. In this paper, since $h$ is analytic these distinction of names do  not matter.
\end{remark}

The function $h$ will take values on the configuration space,
which is $2 \pi \torus^D$.
It is convenient to think of the hull function $h$ as
a real valued function satisfying some periodicity conditions. What is needed is that, when the argument of $h$ changes by
a vector in $ 2 \pi \integer^L$, the output changes by
a vector in $2 \pi \integer^D$. One can easily see that the dependence
of the output should be linear because applying
several translations, the output should add.
This means that the fact that the functions are
from the torus to itself is equivalent to  that
for every $k \in \integer^L$ one should have $h( \th + 2\pi k ) = h(\th) + A k$
with $A$ a $D \times L$ matrix with integer coefficients.
As we will see, in the case that $L =1$, the matrix
$A$ can be identified with a $D$-dimensional integer vector.

\subsubsection{Periodicty in maximal  dimensional tori} 

In the case of maximal dimensional tori ($D = L$) it is natural to
consider $A = \Id$, since this is what happens in the integrable case
and, being a topological property, it has to be preseved under small
perturbations. Hence, we will have
\begin{equation}\label{periodicity2}
  h(\th + \ell ) = h(\th) + \ell  \quad \forall \ell  \in 2 \pi \integer^L.
\end{equation} 
It will be convenient to write \begin{equation}\label{eq:periodic-D2} h(\th) = \th + u(\th)\end{equation} with $u$
a periodic function of period $2\pi$.  We call $u$ the
\emph{periodic part of the hull function}.

\subsubsection{Periodicty in lower dimensional tori} 
In the case of lower dimensional tori ($D > L$)  the situation is richer.
In this paper, we will only consider the case $D=2, L=1$ which
allows to understand the geometry and allows to prove the full results.
In this case, we see that $Ak$ is a $2$-dimensional integer vector, that is: 
\begin{equation}\label{periodicity1}
  h(\th + 2\pi) = h(\th) +  k,
\end{equation}
where $k=(k_1, k_2)\in 2\pi\integer^2$. 

The geometric meaning of $h$ is that it gives a parameterization of
a circle embedding on $\torus^2$ which winds $k_1$ times in the first component
and $k_2$ times in the second direction.  This is a topological property of the
embedding. 
Again, we will write
\begin{equation}\label{periodic-D1}
  h(\th) = \th k  + g(\th)
  \end{equation} 
and call $g$  the periodic part of the hull function. 

We will use
the convention of denoting by $u$ the periodic part of maximal
dimensional tori and by $g$ the periodic part of the lower-dimensional
tori (recall that we will be considering only $1$-dimensional tori in
two degrees of freedom systems).

In the integrable case, the torus in phase space given by $p = \alpha
k $, with $\alpha \in \real$, is invariant but it is also foliated by
$1$-dimensional tori invariant under the map.  It is well known (and
we will verify this perturbatively) that, when $V$ satisfies a
non-degeneracy condition and the frequency is Diophantine, the
perturbations destroy the invariant torus and that only a finite
number of these tori ``persist''.  More precisely, we will show that,
under some conditions verified by \eqref{DSM}, there are formal power
series for a finite number of these tori (both in the symplectic,
$\gamma = 0$, and in the dissipative case, $\gamma \ne 0 $).

\subsection{The invariance equations}

Finding a quasi-periodic sequence $\{ q_n\}$ that is a solution of
\eqref{lagrangian} is equivalent to finding a hull function
which satisfies
\begin{equation} \label{periodicity-hull}
  h(n \omega  + \omega) +  (1 - \gamma \eps^3)  h(n \omega  - \omega)
  - (2 - \gamma \eps^3) h(n\omega)  - \eps V'(h(n\omega))  - \mu = 0
  \quad \forall n \in \integer
\end{equation}

Since $n\omega$ is dense on the torus (we are assuming that $\omega$  is 
Diophantine), and we are assuming that $h$ is continuous,
we see that \eqref{periodicity-hull} is equivalent to
\begin{equation} \label{periodicity-hull-bis}
  h(\th  + \omega) +  (1 - \gamma \eps^3)  h(\th  - \omega)
  - (2 - \gamma \eps^3) h(\th)  - \eps V'(h(\th)) - \mu = 0
  \quad \forall \th \in \torus^L 
\end{equation}

\subsubsection{Underdeterminacy of the invariance equations}
\label{sec:underdeterminacy} 
The  equation \eqref{periodicity-hull-bis}  is underdetermined. If $h$ is solution of \eqref{periodicity-hull-bis}, for any
$\sigma \in \torus^L$  the hull function $h_\sigma$
defined by
\begin{equation}\label{underdeterminacy}
  h_\sigma(\th) = h(\th + \sigma)
\end{equation}
is also a solution.
The underdeterminacy \eqref{underdeterminacy} has
the geometric interpretation of choosing the origin of
the system of coordinates in the reference manifold $\torus^L$.
Clearly, it does not change the range of $h$ nor
the dyamics on it.

For the purposes of this paper, it will be important to choose
a  normalization that fixes the  underdeterminacy. If a normalization
is not fixed, it is impossible to obtain estimates of the coefficients of the expansion.  This normalization will be different in the case of
maximal tori and in the case of lower dimensional tori. 

\begin{remark} 
In other papers \cite{Llave01, LlaveGJV05,CallejaCL13} by taking derivatives along the underdeterminacy in
an approximate  solution one obtains
identities that are useful to simplify the Newton method
and lead to good estimates and efficient numerical methods.
This is reminiscent of the use of Ward identities in gauge theory. Note
that \eqref{underdeterminacy} can be considered as a gauge symmetry.
\end{remark}

\subsubsection{Invariance equations and normalization
  conditions for the periodic parts of the
  hull function}

It will be useful to rewrite \eqref{periodicity-hull-bis}
(and the normalizations) in terms of
the periodic part of the hull function, because in the
calculations of the expansion of the perturbative terms we will be able to use Fourier
analysis.

Since the relation between the hull function and its periodic 
part are different in the case of maximal dimensional tori and in the case
of lower dimensional tori, we will need to deal with these cases separately.

We  also need to supplement the invariance  equations with
normalizations that fix the underdeterminacy \eqref{underdeterminacy} described in
Section~\ref{sec:underdeterminacy}.  These normalizations
are  different
in the case of maximal dimensional tori and lower dimensional tori
and we will discuss them in the next sections. 

\medskip

\leftline{\bf  Maximal dimensional tori} 

Expressing the hull function in terms of its periodic part  \eqref{eq:periodic-D2}
and substituting in \eqref{periodicity-hull-bis} 
we obtain  that for the maximal dimensional tori,
the  invariance equation expressed in terms of the periodic part, $u:\real^D \longrightarrow \real^D$, is:
\begin{equation}
\label{invariance2}
  u(\th + \omega) + u(\th  - \omega) - 2 u(\th) =
  \eps V'(\th + u(\th )) + \mu - \gamma \eps^3( u(\th) - u(\th  -\omega) + \omega) \quad \forall \th \in \torus^D. 
\end{equation}

We emphasize that, since we are assuming $\gamma \ne 0$,
\eqref{invariance2} is an equation for $u$ and $\mu$, which are the
unknowns. The frequency $\omega$ is a data of the problem. We recall
that in this case $L=D$.

The  underdeterminacy \eqref{underdeterminacy}
for the hull function can be translated for the periodic part as follows:  If $u, \mu$ is a solution of \eqref{invariance2}
so is $u_\sigma, \mu$
where
\begin{equation}\label{usigma} u
  _\sigma(\th) = u(\th + \sigma) + \sigma.
\end{equation} 
To settle the underdeterminacy, we have found convenient to use
the normalization: 
\begin{equation} \label{normalizationmax}
  \int_{\torus^D} u(\th) \, d\th = 0.
\end{equation}
It is easy to check that given any periodic function $u$, there is a unique $u_\sigma$ as in \eqref{usigma}
satisfying the normalization \eqref{normalizationmax}. In \cite{CallejaCL13} it is shown that  the solutions of
the invariance equation \eqref{invariance2}  with the normalization
\eqref{normalizationmax} are locally unique.
In this paper, we will show that the power series expansions
solving \eqref{invariance2}, \eqref{normalizationmax}
are also unique.

\medskip 
\leftline{\bf Lower dimensional tori}

For simplicity, we will only consider
the case $D=2, L= 1$. To consider different cases for $D$ and $L$ would require different geometric considerations.

If we substitute \eqref{periodic-D1} into
\eqref{periodicity-hull-bis} we obtain the following equation for the periodic part, $g:\real \longrightarrow \real^2$, 
\begin{equation}\label{invariance-lower}
    g(\th+\omega) - 2g(\th) + g(\th - \omega)
  = \eps V' (\th k + g(\th) ) + \mu - \gamma\eps^3\left( k\omega +g(\th) - g(\th -\omega)   \right)
\end{equation}

We also note that the underdeterminacy described
Section~\ref{sec:underdeterminacy} translates into
the fact that if $g, \mu$ are solutions of
\eqref{invariance-lower},
then, for any $\sigma \in \real$,
$g_\sigma, \mu$ are also solutions, where
\begin{equation} \label{gsigma}
  g_\sigma(\th) =  \sigma k +  g(\th + \sigma). 
\end{equation} 

The normalization we have found convenient to take
is
\begin{equation}
  \label{normalizationlow}
  \int_0^{2\pi}g(\th) \cdot k  \, d \th = 0
\end{equation}
In this paper, we will show that after we make some choices for the
first order of the expansion, all the other terms of
the formal power series solution of \eqref{invariance-lower} satisfying
\eqref{normalizationlow} are unique. We will also show that  there
are always several solutions for the equations at low order.
The existence of several solutions corresponds to the well known
fact that a resonant torus breaks up into a finite number of
low dimensional tori.  See \cite{Treshev} for the Hamiltonian
case. We will verify this in the sense of formal power series. 

\subsection{Some Standard Definitions and some notations} 
\label{sec:definitions} 
In this section we recall some standard definitions 
and introduce some small typographical notations. 

\begin{definition}\label{diophantine}
  We say that $\omega  \in \real^L$ is Diophantine, of type $(\nu, \tau)$, 
  if  there exist $\nu, \tau > 0 $ such that
\begin{equation} \label{Diophantine-standard} 
    | m \cdot \omega -  n |^{-1}  \le \nu |m|^{\tau}\quad n\in \integer, m \in
    \integer^L \setminus \{0\}
\end{equation} 
where $|m| =\sum_{j=1}^L |m_j|. $
\end{definition}

\subsection{Some typographical notation} 
\label{sec:notation}
We will introduce some notation that will allow
to carry out some calculations more concisely: 
\begin{equation}\label{notation}
  \begin{split} 
   \LL_\omega[u](\th) &\equiv u(\th + \om) + u(\th -\om) - 2u(\th) \\
   \E^\om_\eps[u, \mu ](\th) &\equiv  \LL_\omega[u](\th) -   \eps V'(\th + u(\th )) - \mu + \gamma \eps^3( u(\th) - u(\th  -\omega) + \omega) \\
   \E^{\om, k}_\eps[g, \mu ](\th) &\equiv  \LL_\omega[g](\th) -   \eps V'(\th k + g(\th )) - \mu + \gamma \eps^3( g(\th) - g(\th  -\omega) + \omega k)
\end{split}
\end{equation}

So that \eqref{invariance2} and  \eqref{invariance-lower}  can be written more concisely as
\[
\E^\omega_\eps [u, \mu] = 0 \quad\mbox{and}\quad  \E^{\omega, k}_\eps [g, \mu] = 0
\]
respectively. Note that, even if the formulas defining $\E_\eps^\omega$ and $\E_\eps^{\omega,k}$ are very similar, the functions in which the operators act are very different
(e.g. they have domains with  different dimensions and they have different
normalization conditions).

\begin{definition}\label{def:formalsolutions}
  We fix a  norm $\| \cdot \|$ in the space of periodic
  functions. We say that the formal series
  $u  = \sum_n \eps^n u_n$, $\mu = \sum_n \eps^n \mu_n$
  is a  solution of \eqref{invariance2} in the sense of
  formal power series if for every $N \in \nat$,
  there exists $C_N \in \real^+$ such that  
  \[
    \left\| \E^\omega_\eps\left[ \sum_{n=1}^N \eps^n u_n, \sum_{n=1}^N \eps^n\mu_n\right] \right\| 
\le C_N |\eps|^{N+1}.
  \]
  Similarly, we say that the formal series $g= \sum_n \eps^n g_n$, $\mu = \sum_n \eps^n\mu_n$ is a solution of \eqref{invariance-lower} in the sense of formal power series if for every $N\in \nat$, there exists $\tilde C_N\in\real^+$ such that $$ \left\| \E^{\omega,k}_\eps\left[ \sum_{n=1}^N \eps^n g_n, \sum_{n=1}^N \eps^n\mu_n\right] \right\| 
\le \tilde C_N |\eps|^{N+1}.$$
\end{definition}

Of course, when we seek solutions in formal power series we
also impose that the solutions satisfy the normalization conditions.
To get series that satisfy \eqref{normalizationmax} or \eqref{normalizationlow}
it is equivalent that the coefficients satisfy
\[
\int_{\torus^D} u_n (\th)\, d\th = 0, \qquad \int_0^{2\pi} k \cdot g_n(\th) \, d\th = 0,
\]
respectively, where $k\in 2\pi \integer^2$ is fixed by the chosen hull function as in \eqref{periodic-D1}. 

\medskip

The definition of formal solutions
depends on  the norm considered. Some formal power series may be a series in one norm but not in another.  Note however  that a formal power series in a norm is also a formal power series in all
the weaker norms.

Note that we are not even assuming that $\| u_n \|$ is finite for
every $n$. In principle, we could have that the $u_n$ lie in
a wider space but that they cancel in the evaluation of
the $n$ order term (we, of course, need that the operator $\E^\omega_\eps$
can be defined). 

Formal power series are not meant to converge.
Hence, rearrangements, etc, can alter  their prroperties.
Nevertheless, formal power series can be added, multiplied (using the
Cauchy formula for convergent power series), 
derived with respect to $\eps$, etc. in such a way that many of
the standard rules of these operations apply to the algebra of formal power series,
\cite{Cartan63}.

Note that in this paper we are not dealing with formal power series
in general, but rather with formal power series that solve a
functional equation.
Formal power series  that satisfy a well behaved equation (specially if there is an
a-psteriori theorem for this equation) enjoy many more properties than
general formal power series.

\begin{definition}\label{gevrey}
 We fix a  norm $\| \cdot \|$ in the space of periodic
 functions. We say that a formal power series, $\sum_{n=0}^\infty \eps^n u_n$, is Gevrey  if for any $n\in \nat$ there
 exist $A, R, \sigma > 0$ such that: 
 \begin{equation} \label{gevreybounds}
   \| u_n \| \le A R^n (n!)^\sigma
 \end{equation} 
\end{definition}

\begin{remark}
  The Definition~\ref{gevrey} depends on the norm considered. If a series is Gevrey for a norm is also Gevrey for all the weaker norms.
\end{remark}

\begin{remark}
  The factor $R^n$ in Definition \ref{gevrey}  is subdominant with
  respect to the factorial. It can be eliminated just
  by rescaling the parameter $\eps$ in the power series.

  We have included the factor
  $R^n$ in the definition to be compatible with the literature, but in this
  paper, one can take $R=1$ and then, reintroduce it by scaling the parameter.

  \end{remark}
\begin{remark} 
As it is well known, formal power series do
not need to converge for any value of $\eps$.
Nevertheless, for small values of $\eps$, the  sum
may be a very good approximation for the solution for 
small values of the perturbation parameter.  For a small value of 
$\eps$ adding a few more terms increases the accuracy, but 
if we add more and more, we start seeing the divergence of the series. A practical problem is to know what is the optimal number of terms to sum, given a certain $\eps$. Some estimates on
the optimal term  to stop can be obtained from estimates
on the size of the coefficients $C_N$ or on the growth of
the coefficients (notably Gevrey).

\end{remark}

Note that the definition of a solution in the sense of formal power
series and the Definition \ref{gevrey}  depend on the choice of
a norm in the space of periodic functions. The norms
that we will use in this paper are the following.

\begin{definition}\label{norms}
  Given $\rho \ge  0, r\in \integer_+$ and a periodic function $u:\real^L \longrightarrow \complex^D$  with Fourier expansion $u(\th) = \sum_{\ell\in\integer^L} \hat{u}_\ell e^{i\th\cdot \ell}$, we define

  \begin{equation}\label{eq:norms}
    \| u \|_{\rho, r}^2 = |\hat u_0|^2 + \sum_{\ell\neq 0}      | \hat u_\ell|^2 e^{2|\ell|\rho}(1 + |\ell|^{2})^r 
  \end{equation}
\end{definition}

Note that, in the case of maximal dimensional tori, since we will be looking for functions satisfying the normalization \eqref{normalizationmax}, 
we can ignore the zero Fourier coefficient for these functions.

The usefulness of the norms \eqref{eq:norms} is that the size of
the norms can be read off easily from the size of the Fourier
coefficient and, at the same time, they satisfy the Banach
algebra property under multiplication when $r>d/2$ (See \cite{XuLW22}), that is
\[
  \| u v \|_{\rho, r} \le C \| u \|_{\rho, r} \| v \|_{\rho,r}.
\]
When $\rho > 0$, the functions with a finite norm are analytic
functions in a strip around the real torus.  When $\rho = 0$,
they become  the standard Sobolev norms.  These norms are
particular cases of Bergman spaces in domains.  Even if
we will not use it here, we note that these norms come from
an inner product, so the spaces are Hilbert spaces and
one can use orthogonal projections, etc.

\begin{remark}\label{projection} 
Given a function $V$, subtracting from
it is average reduces the norm, that is
\[
  \| V - \mu\|_{\rho, r} \le \| V \|_{\rho,r}
  \quad \mbox{with}\quad \mu = \int V(\th) \, d\th.
\]
This will be used in our calculations when we take $\mu$ to
be the correction of the drift. Due to the fact that  the spaces we are considering are Hilbert spaces,  the removal of 
the average     is an orthogonal projection on the space of 
functions with zero average and hence, does not increase the norm. 
\end{remark} 

\begin{remark} 
The norms $\| \cdot \|_{\rho, r} $ are not equivalent, 
but this do not seem to make any difference in our results. Note that for $\delta > 0$, 
\[
\| f \|_{\rho - \delta, r'} \leq C 
\| f \|_{\rho,  r} 
\] 
These spaces also satisfy  interpolation inequalities in $\rho$ and $r$. 
\end{remark}

\section{Statement of results}
\label{sec:statement}

The main result of this paper is the following 
\begin{theorem}\label{thm:main}
  We consider three problems \\
\medskip
\textbf{Maximal dimensional tori}
\begin{itemize}
    \item 
    Formal power series solutions  $u, \mu$ of \eqref{invariance2} satisfying also \eqref{normalizationmax} for any $D$. We assume $\gamma \ne 0$.
\end{itemize}
\textbf{Lower dimensional tori}
\begin{itemize}
    \item
  Formal power series solutions  $g, \mu$ of \eqref{invariance-lower} satisfying
  also \eqref{normalizationlow} for $D = 2, L =1 $, 
  We assume $\gamma \ne 0$.
\item
  Formal power series solutions  $g, \mu$ of \eqref{invariance-lower} satisfying
  also \eqref{normalizationlow} for $D = 2, L =1 $, 
  We assume $\gamma  =  0$.
 \end{itemize}

  In all cases we assume that $V'$ is a trigonometric polynomial of degree $J$
  (i.e. it satisfies \eqref{polynomial}).

  For the case of lower dimensional tori (we only deal with $D=2, L = 1$) 
  we need to specify an integer vector $k$ giving the topology
  of the embedding of the torus, and fix another integer  vector $k^\perp$
  orthogonal to $k$. 

  We choose  $\beta_0 \in [0,2\pi) $  such that
  $   \int_0^{2\pi} k^\perp\cdot V'(\th k + \beta_0  k^\perp  )d\th = 0$
  (we will show that there are always at least two choices of such $\beta_0$). 
  
  We assume that such $\beta_0$ is such that 
  \begin{equation} \label{lowerassumption}
    \int_0^{2\pi} k^\perp \cdot  D^2 V(\th k + \beta_0  k^\perp )k^\perp d\th \neq 0
    \end{equation}

  Then, we have: \\
  \textbf{Maximal dimensional tori}\\
  A) If $\omega\in \real^D$ is an irrational multiple of $2 \pi$, then there exists
  a unique solution $u_\eps = \sum \eps^n u_n$,  $\mu_\eps = \sum  \eps^n\mu_n$ of \eqref{invariance2} satisfying  \eqref{normalizationmax} in the sense of formal power series (Definition~\ref{def:formalsolutions}). Furthermore, each $u_n$  is a trigonometric polynomial of degree $nJ$.

  B) If $\omega$ is a Diophantine multiple of $2 \pi$ for any $\rho \ge 1$,
  $r > D/2$, the formal solutions in Part A) are Gevrey (Definition~\ref{gevrey})
  under  the norm $\rho, r$ defined in \eqref{norms}. That is, there exist $A,R,\sigma$ such that
  \[
    \| u_n \|_{\rho, r}, |\mu_n| \le A R^n (n!)^\sigma
  \]
\textbf{Lower dimensional tori}\\
A) If $\omega\in \real$ is an irrational multiple of $2 \pi$, then there exists a unique 
  $g_\eps = \sum \eps^n g_n$,  $\mu_\eps = \sum  \eps^n\mu_n$  solution of \eqref{invariance-lower} satisfying  \eqref{normalizationlow} in the sense of formal power series (Definition~\ref{def:formalsolutions}). Furthermore, each $g_n$  is a trigonometric polynomial of degree $nJ$.

  B) If $\omega$ is a Diophantine multiple of $2 \pi$ for any $\rho \ge 1$,
  $r > 1/2$, the formal solutions in Part A) are Gevrey
  under  the norm $\rho, r$ defined in \eqref{norms}. That is, there exist $A,R,\sigma$ such that
  \[
    \| g_n \|_{\rho, r}, |\mu_n| \le A R^n (n!)^\sigma
  \]
\end{theorem}

The two parts of Theorem~\ref{thm:main} are rather different.
Part A) is purely algebraic, Part B)  requires estimates. Actually part A) is based  on an explicit algorithm that
has been repeatedly used.  In the case of
maximal dimensional tori  and $D=1$, it was implemented as a 
computer program and run in
\cite{BustamanteC19, BustamanteC21}. 

\begin{remark} \label{nopolynomial}
  Along the proof of part A), we will see that the assumption that
  $V$ is a trigonometric polynomial can be weakened to $V$ analytic
  provided one assumes that $\omega$ satisfies  the following  weak Diophantine
  condition:

  For all $\delta > 0$ we have 
  \begin{equation} \label{weakdiophantine}
 \lim_{n \to \infty}  \sup_{\ell \in \integer^D \setminus \{0\}}  |\omega\cdot \ell - n |^{-1} e^{- \delta n }  = 0 
  \end{equation}
\end{remark}

As we will see in the proof, the precise meaning of the uniqueness
statement in part A) is very strong.  The solution is unique in the
sense of power series with more general coefficients.  In 
\cite{CallejaCL20} there are local uniqueness results (both in the
conservative and conformally symplectic cases) which are not
perturbative.

\section{Proof of Part A) of Theorem ~\ref{thm:main} }
\label{sec:partA} 

The proof of Part A) is based on an algorithm.
This algorithm has indeed been used many times for several problems.
For the problems described here it has been used in
\cite{BustamanteC19,BustamanteC21, Jo-Lla-Zou-99}.   The proof of
part B) will be based on estimating in detail steps of
the argument. 

We will start by collecting  some more or less standard ingredients of
the argument. 

\subsection{The inverse of the operator $\LL_\omega$}

The operator $\LL_\omega$, defined in \eqref{notation}, is
diagonal in Fourier series.  If
\begin{equation} \label{Fourier}
  u(\th) = \sum_{\ell\in\integer^L} \hat u_\ell e^{i \ell \cdot \th},
\end{equation}
(where the sum can be understood in many meanings, including
$L^2$, but in this paper finite sums will be enough), then: 
\begin{align}
     \LL_\omega[u](\th) = \sum_\ell m_\ell  \hat u_\ell e^{i \ell \cdot \th} \\
      m_\ell = 2(\cos(\ell \cdot \omega) -1 ) \label{m_l}
\end{align}

  We observe that if $\omega$ is an irrational multiple of $2 \pi$
  the multiplier, $m_\ell$, of the Fourier coefficients is
  not zero except for $\ell = 0$. We note that in this paper we will need to invert the operator only in trigonometric 
  polynomials. The following result is trivial, but we
  state it to be able to refer to it. 

  \begin{lemma}\label{lem:inverse}
    Let $B:\real^L \longrightarrow \real^D$ be a trigonometric polynomial of degree $J$. That is: 
    \[
      B(\th) = \sum_{ |\ell| \le J } \hat B_\ell  e^{i  \ell\cdot \th}
    \]
If $\hat B_0 = 0 $ and $\omega$ is an irrational multiple of $2\pi$,  then there exists a $L^2$ function, $A$, 
    solving
    \[
      \LL_\omega A = B.
    \]
    The solution is unique if it also satisfies the normalization  \eqref{normalizationmax}. 
    Furthermore, $A$ is a trigonometric polynomial of degree $J$ and for any $\rho\ge 1, r > L/2$ we have
    \[
      \| A \|_{\rho, r} \le \max_{0 < |\ell | \le J} |m_\ell |^{-1} \| B \|_{\rho,r}. 
    \]
    where $m_\ell$ is defined in \eqref{m_l}. In particular, if $\omega$ is Diophantine multiple of $2 \pi$, of type $(\nu, \tau)$,
    we have
    \begin{equation}\label{goodbounds}
      \| A \|_{\rho, r} \le 4  \nu^{-2} J^{2 \tau }     \| B \|_{\rho,r}.
    \end{equation}
  \end{lemma}


  \subsection{Computation of formal power series of trigonometric functions.}  \label{sec:trigonometricpower}

  We consider first the case of maximal dimensional tori. Given a
  formal power series $ u_\eps(\th) = \sum_n \eps^n u_n(\th)$, $\th\in
  \torus^D$, there are efficient ways to compute the series expansion
  of $V'(\theta + u_\eps(\theta))$ on the right hand side of
  \eqref{invariance2}.

  Since we are assuming that $V'$ is a trigonometric polynomial,
  it is enough to describe the
  computation of the terms
  \begin{equation}\label{setup}
      e^{i\ell\cdot(\th + u_\eps(\th))}
  \end{equation}
  in \eqref{polynomial}.  The idea goes back to \cite{BrentK78}
  \cite[Sec. 4.7]{Knuth} and it has been used many times in
  \emph{automatic differentiation} literature.  The same ideas can be
  used for many other functions besides exponential functions.

If we write: 
  \begin{equation} 
  \begin{split} \label{expansion_ek}
   & e^{i\ell\cdot (\th + u_\eps(\th))} = \sum_{n=0}^\infty \eps^n E_n^\ell(\th) \\
 \end{split}
\end{equation}
and take a derivative with respect to $\eps$ in \eqref{setup}  we obtain
  \begin{equation} \label{chain} 
    \begin{split}
    & \frac{d}{d \eps} e^{i\ell\cdot (\th + u_\eps(\th))} =  i\ell \cdot \frac{d}{d\eps} u_\eps (\th)  e^{i\ell\cdot (\th + u_\eps(\th))}.
   \end{split}
\end{equation}
Using the notation for the expansions in \eqref{expansion_ek}, and the product formula for
power series, we obtain by equating coefficients of order $\eps^{n-1}$
in \eqref{chain},  that for $n\geq 1$
  \begin{equation} \label{recursion} 
    \begin{split}
    & n  E^\ell_n (\th) = \sum_{m = 0}^{n-1}  (m+1)i\ell\cdot u_{m+1}(\th) E^\ell_{n -1 -m} (\th)  \\
 \end{split}
\end{equation}
The equation \eqref{recursion} allows to compute
$  E^\ell_n (\th)$
given that we have
$  E^\ell_0 (\th), \ldots,  E^\ell_{n-1}(\th)$  and $u_0(\th), \ldots, u_n(\th)$. Note that $E_0^\ell(\theta) = e^{i\ell \cdot \theta}$.

\medskip

The case of lower dimensional tori is rather similar. We consider a formal power series $g_\eps(\th) = \sum \eps^n g_n(\th), \th\in\torus^L$ and seek to compute the formal expansions of \begin{equation}
    e^{i\ell\cdot(\th k + g_\eps(\th) )} = \sum_{n=0}^\infty \eps^n F^{\ell, k}_n(\th). 
\end{equation}
Analogously we obtain the following relations \begin{equation}\label{eq:recursion-Froeshle}
    n F_n^{\ell, k}(\th) = \sum_{m=0}^{n-1}(m+1) i\ell \cdot g_{m+1}(\th) F_{n-1-m}^{\ell, k}(\th).
\end{equation}
with $F_0^{\ell,k}(\theta) = e^{i \ell\cdot (\theta k + g_0)}$ and $g_0$ a constant to be determined in Section \ref{sec:recursive-lower}. Note that
\eqref{recursion} and \eqref{eq:recursion-Froeshle} have the same
shape, however we have decided to use different notation to
distinguish between the maximal and the lower dimensional tori. We
note that recursions \eqref{recursion} and
\eqref{eq:recursion-Froeshle} are easy to implement numerically even
with extended precision numbers.

\subsection{Recursive solution  of \eqref{invariance2}}
\label{sec:recursive}

Writing $u =\sum \eps^n u_n $ and equating terms of order $\eps^n$ in both sides of \eqref{invariance2},
we obtain that solving \eqref{invariance2} in the sense of formal power
series is equivalent to solving the sequence of equations (indexed by
$n$):

\begin{equation}\label{mainrecursion} 
    \LL_\omega u_n(\th) = \sum_{|\ell|\leq J} \alpha_\ell E_{n-1}^\ell(\th) +\mu_n \quad\mbox{for } 1\leq n \leq 2 
\end{equation}

\begin{equation}\label{eq:rec-3}
    \LL_\omega u_3(\th) = \sum_{|\ell|\leq J} \alpha_\ell E_2^\ell(\th) -\gamma \omega + \mu_3
\end{equation}

\begin{equation}\label{eq:cohom-eq-order-n}
    \LL_\omega u_n(\th) =\sum_{|\ell|\leq J} \alpha_\ell E_{n-1}^\ell(\th) + \mu_n  - \gamma u_{n-3}(\th) + \gamma u_{n-3}(\th-\omega) \quad\mbox{for } n \geq 4
\end{equation}
where the coefficients $E_n^\ell(\theta)$ are given by the recursion \eqref{recursion}. 
\begin{remark}\label{rem:ord-0-max}
    The equation at order $\eps^0$ is $\LL_\omega u_0(\theta) = \mu_0$ and its normalized solution is given by $u_0(\theta) \equiv 0$, $\mu_0 = 0$. 
\end{remark}
We will show by induction in $n$ that all the equations \eqref{mainrecursion}, \eqref{eq:rec-3}, and \eqref{eq:cohom-eq-order-n}
for $u_n$ 
can be solved in a unique way (we will make some precisions on
the uniqueness).

\medskip 

If we assume that $u_0, \ldots, u_{n-1}, E_0^\ell,\ldots, E_{n-2}^\ell$ are known, by the relations
described in Section~\ref{sec:trigonometricpower}, we have
that the  $ E^\ell_{n-1}$ are known.

Furthermore, if
$u_0, \ldots, u_{n-1}, E_0^\ell,\ldots, E_{n-2}^\ell$ are trigonometric
polynomials, so is  $E^\ell_{n-1}$ and therefore also the right hand sides of \eqref{mainrecursion}, \eqref{eq:rec-3}, \eqref{eq:cohom-eq-order-n} (a more precise
version will be obtained in Proposition~\ref{trigonometric} and
we will obtain also estimate on the degree). 
In view of justifying Remark~\ref{nopolynomial}, we point out that
since the recursions are algebraic operation, we can use analyticity
instead of trigonometric polynomials in the argument above.

Hence, all the R.H.S of
\eqref{eq:cohom-eq-order-n} is known except for the number  $\mu_n$. We choose $\mu_n$ in such a way that the R.H.S. has zero average, that is \begin{equation}\label{eq:mu-max-tori}\mu_n = -\int_{\torus^D}  \sum_{|\ell|\leq J} \alpha_\ell E_{n-1}^\ell(\th)\, d\th \end{equation}

Hence, applying Lemma~\ref{lem:inverse}
we  find a unique $u_n$ with zero average as required by  normalization.

Thus, we have proved  that provided that the perturbative expansion is
known up to order $n-1$, we can find a unique normalized $u_n$.
The first cases of the induction are easy to compute by hand.

\begin{remark}
  Note that the uniquess statement comes from the application of
  Lemma~\ref{lem:inverse}, so the solutions are unique even
  in classes of formal power series expansions of functions in wider
  classes than analytic. 

  Even if we allowed that the $u_n$ were $L^2$ functions -- or even less
  regularity -- we would obtain that they are unique and
  that they are trigonometric polynomials.
\end{remark}

\begin{proposition}
    \label{trigonometric} 
  Assume that $V' (\th) =  \sum_{|\ell|\leq J} \alpha_\ell e^{ i \ell \cdot \th}$ is a trigonometric polynomial of degree $J$. Then, $u_n$ is trigonometric polynomial of degree at most $nJ$ and the coefficients $E_n^\ell$ in the expansion \eqref{expansion_ek} are trigonometric polynomials of degree at most $(n+1)J$. Hence,  $\sum_{|\ell|\leq J} \alpha_\ell E_n^\ell$ is a trigonometric polynomial of degree at most $(n+1)J$.
  \end{proposition} 

  \begin{proof}
    We prove the result by induction in the index $n$. We assume that the conclusion is true  for
    $m\leq n-1$, that is $u_{m-1}$ and $ E^\ell_{m-1}$ are trigonometric polynomials of degree less or equal to $(m-1)J $ and $m J$, respectively, for $m= 0,  \ldots,  n-1$.

    Since  $\sum_{|\ell|\leq J} \alpha_\ell E_{n-1}^\ell (\th)$ is a trigonometric polynomial of degree at most $nJ$, then the RHS of  \eqref{eq:cohom-eq-order-n} is a trigonometric polynomial of the same degree. Therefore, $u_n$ has at most degree $nJ$. 
    
    Moreover, using the recursion \eqref{recursion} we obtain that, under the inductive
    assumption, each of the terms in the sum  in the RHS of  recursion \eqref{recursion} are of
    degree  at most $(m +1 )J + (n-m)J = (n+1)J $. Hence, the sum  $\sum_{|\ell|\leq J} \alpha_\ell E_{n}^\ell(\th)$ is a trigonometric polynomial of degree less or equal than $(n+1)J$. 

    The initial cases for $m =0$ are  automatic because $u_0\equiv 0$ and $E_0^\ell(\th) = e^{i\ell\cdot\th}$.

\end{proof} 

\subsection{Recursive solution  of \eqref{invariance-lower}}
\label{sec:recursive-lower}
For the case of lower dimensional tori we treat separately the conservative case ($\gamma=0$) and the dissipative case ($\gamma\neq 0$).

\subsubsection{Low dimensional tori, Conservative case ($\gamma =0$) }\label{sec:sec:recur-low-conserv}

In this case the drift parameter $\mu$ is not necessary, thus
$\mu=0$. Then, we only consider the formal series $g_\eps(\th) =
\sum_{n=0}^\infty g_n(\th)\eps^n$, where
$g_n:\real\longrightarrow\real^2$. To solve \eqref{invariance-lower}
in the sense of formal power series is equivalent to solve the
following sequence of equations:\begin{equation}\label{eq:eq-ord-n}
  g_n(\th + \omega) - 2g_n(\th) + g_n(\th-\omega) = R_n(\th),
\end{equation}
where $R_n$ in \eqref{eq:eq-ord-n} are the coefficients of the series $ \eps V'(\th k + g_\eps (\th) ) = \sum_{n=0}^\infty R_n(\th)\eps^n $.  We note that $R_n(\theta)$ only depends on $g_0, g_1, ..., g_{n-1}$. 

Equations \eqref{eq:eq-ord-n} can be solved as long as the right hand side has zero average, see Lemma \eqref{lem:inverse}. As it happens often in
Lindstedt series, the zero order term and the first order
terms will require different considerations (indeed, the first
order term requires a new non-degeneracy condition). Then, all
the terms of higher order can be done in the same way.
As it also happens often in Lindstedt series, the general step
proceeds by assuming that $g_0,\ldots, g_{n-2} $ are completely known and
$g_{n-1}$ is known up to a constant. Then, we determine the constant in
$g_{n-1}$  to ensure the solvability of an equation which determines $g_n$ up to an additive constant.

The zero order equation is
\begin{equation}\label{eq:g_0}
  g_0(\th + \omega) - 2g_0(\th) + g_0(\th -\omega) = 0 \end{equation} which is solved by any constant $g_0$.
We will chose $g_0$ orthogonal to $k$ to satisfy the normalization
\eqref{eq:norm-low-gn}. That is,  we will set $g_0 = \beta_0 k^\perp$ for
some $\beta_0 \in \real$. As we will see, the $\beta_0$ will be
fixed so that the equation of order $1$ is solvable.

In summary, the equation of order $0$ determines $g_0$ up to
a constant, that will be determined through the analysis of the equation of
order $1$.

The equation of order $\eps^1$ becomes
\begin{equation} \label{eq:order-1}
    g_1(\th + \omega) - 2g_1(\th) + g_1(\th - \omega) = 
V'( \th k + \beta_0 k^\perp)
\end{equation}

The condition for existence of $g_1$ solving  \eqref{eq:order-1} 
is $\int_0^{2\pi} V'(\th k +\beta_0 k^\perp )d\th  = 0$,
which is equivalent to the two conditions
\begin{equation} \label{eq:condition-ord-1}
  \int_0^{2\pi} k\cdot V'(\th k +\beta_0 k^\perp )d\th = 0, \qquad \int_0^{2\pi} k^\perp\cdot V'(\th k + \beta_0 k^\perp )d\th = 0.
\end{equation}
The first condition in \eqref{eq:condition-ord-1} is always satisfied for any
$\beta_0$
because
\[k\cdot V'(\th k + g_0) = \frac{d}{d\th} V(\th k + \beta_0 k^\perp).\]

For the second condition in \eqref{eq:condition-ord-1}
we observe that
\begin{equation}\label{eq:existance-beta0}
    \int_0^{2\pi}\int_0^{2\pi} k^\perp\cdot V'(\th k + \beta_0 k^\perp) d\th d\beta_0 =\int_0^{2\pi} \int_0^{2\pi}\frac{d}{d\beta_0} V(\th k + \beta_0 k^\perp) d\beta_0 d\th = 0.
\end{equation}
Since $\int_0^{2\pi} k^\perp\cdot V'(\th k + \beta_0 k^\perp)d\th$ is a
periodic function of $\beta_0$ and its average on $\beta_0$  is zero, there exists at least two values of $\beta_0$ for which the second condition in
\eqref{eq:condition-ord-1} is satisfied.

Notice that choosing $\beta_0$ completes the determination of
the term $g_0$. Then, using Lemma~\ref{lem:inverse} on the invertibility of $\LL_\omega$,
we can determine $g_1$ up to an additive constant vector $\hat g_1^0$. 
Once we choose $\beta_0$ (as we have shown there are at least two
choices, possibly more) we impose the nondegeneracy condition
\eqref{lowerassumption}.

The component of  $\hat g_1^0$ along the $k$ direction 
will be set to $0$ so that the normalization \eqref{normalizationlow} is satisfied.
The component of  $\hat g_1^0$ along $k^\perp$ will be determined
in the next order equation.

This pattern will be repeated.
The equation of order $n$ will be a difference equation for $g_n$.
We inductively assume that at this stage $g_0, \ldots, g_{n-2}$ is
determined completely and $g_{n-1}$ is determined up to a constant.
By requiring that the equation for $g_n$ is solvable we will determine
the constant in $g_{n-1}$ and, by solving the equation, we will determine
$g_n$ up to an additive constant. That is, we will get the same situation
as in the beginning, but with $n-1$ replaced by $n$.
As we will see, this will require a non-degeneracy assumption
\eqref{lowerassumption}. The determination of the constant equation
in $g_0$ is a non-linear equation, but in the higher order
equations, we get that the equation to be solved to get the constant is
a linear equation. The fact that
the terms are determined up to a constant, which is determined in the next
step is very common in Lindstedt series.

\begin{lemma}[\cite{Jo-Lla-Zou-99}]\label{lem:constants-lower-dim}
  Let $n\geq 2$. Let $g_0 = \beta_0 k^\perp$ the choice  for the solution of the equation \eqref{eq:g_0} satisfying \eqref{eq:condition-ord-1}.  We assume
  that it also satisfies the nondegeneracy condition  \eqref{lowerassumption}.

  If we have that $g_0, g_1, g_{n-1} + \beta_{n-1} k^\perp$  solve  the invariance equation \eqref{eq:eq-ord-n}  of order $0,1, \cdots n-1$, respectively, 
  as well as the normalization \eqref{normalizationlow} (the $g_0,\ldots, g_{n-1}$ are uniquely determined, $\beta_{n-1}$ is
  arbitrary), then:
  \begin{itemize}
    \item
    There is only one  $\beta_{n-1}$ so that the invariance
    equation \eqref{eq:eq-ord-n}, of order
    $n$,  has a solution.
  \item
    All the solutions of the equation \eqref{eq:eq-ord-n} of order
    $n$ and the normalization \eqref{normalizationlow} are of the form
    \[
    g_n + \beta_{n} k^\perp
    \]
    where $g_n$ is uniquely determined and $\beta_n$ is arbitrary.
   \end{itemize}   
\end{lemma}

\begin{proof}
Since $R_n(\theta)$ is a trigonometric polynomial, see Proposition \ref{prop: Froeshle-trigonometric} below, equation \eqref{eq:eq-ord-n} has a solution if $\int_\torus R_n(\theta) d\theta = 0 $, see Lemma \ref{lem:inverse}. The last condition is equivalent to the conditions \begin{equation}\label{eq:zero-R-n}
    \int_0^{2\pi} k\cdot R_n(\theta) d\theta
 = 0, \qquad \int_0^{2\pi} k^\perp \cdot R_n(\theta)d\theta = 0. 
 \end{equation}

First we check that the first condition in \eqref{eq:zero-R-n} is always satisfied. Introducing the notation $g_\eps^{[<n] }(\theta) := \sum_{j=0}^{n-1} g_j(\theta)\eps^j $, by definition we have \begin{equation} \label{eq:inv-low-sum}
    \LL_\omega g_\eps^{[<n]}(\theta) - \eps V'(\theta k+ g_\eps^{[<n]}(\theta) ) = R_n(\theta)\eps^n + O(|\eps|^{n+1}).
\end{equation}
Hence, taking the scalar product of \eqref{eq:inv-low-sum} with $ k+ \frac{d}{d\theta}g_{\eps}^{[<n]} (\theta)$ and integrating one obtains  \begin{align}
    0 &= \int_0^{2\pi} k\cdot \LL_\omega g_{\eps}^{[<n]}(\theta) d\theta + \int_0^{2\pi} \frac{d}{d\theta}g_{\eps}^{[<n]}(\theta)\cdot \LL_\omega g_{\eps}^{[<n]}(\theta)d\theta \label{eq:expr-R-zero-aver} \\
    &\quad -\eps \int_0^{2\pi} \left( k + \frac{d}{d\theta}g_{\eps}^{[<n]}(\theta) \right)\cdot V'(\theta k + g_{\eps}^{[<n]}(\theta))d\theta \nonumber\\
    &\quad - \eps^n \int_0^{2\pi} \frac{d}{d\theta} g_{\eps}^{[<n]}(\theta)\cdot R_n(\theta)d\theta     \nonumber\\
    &\quad - \eps^n \int_0^{2\pi} k\cdot R_n(\theta) d\theta    \nonumber\\
    &\quad + O(|\eps|^{n+1}).\nonumber
\end{align}
Note that, since $V$ is a periodic function one has \begin{equation}
    \int_0^{2\pi} \left( k + \frac{d}{d\theta}g_{\eps}^{[<n]}(\theta) \right)\cdot V'(\theta k + g_{\eps}^{[<n]}(\theta))d\theta = 0.
\end{equation}
Moreover, $\int_0^{2\pi} k\cdot \LL_\omega g_{\eps}^{[<n]}(\theta)d\theta = 0 $ because for any periodic function, $f$, one has $$ \int_0^{2\pi}k\cdot f(\theta + \omega)d\theta = \int_0^{2\pi} k\cdot f(\theta)d\theta = \int_0^{2\pi} k\cdot f(\theta - \omega)d\theta.$$
One also has that  $\int_0^{2\pi} \frac{d}{d\theta}g_{\eps}^{[<n]}(\theta)\cdot \LL_\omega g_{\eps}^{[<n]}(\theta)d\theta$ = 0, because $$ \int_0^{2\pi} \frac{d}{d\theta} g_{\eps}^{[<n]}(\theta)\cdot g_{\eps}^{[<n]}(\theta) d\theta = \int_0^{2\pi}\frac{1}{2} \frac{d}{d\theta}\| g_{\eps}^{[<n]}(\theta)\|^2 d\theta = 0 $$ and \begin{align*}
    \int_0^{2\pi} \frac{d}{d\theta} g_{\eps}^{[<n]}(\theta)\cdot g_{\eps}^{[<n]}(\theta + \omega) d\theta  &= - \int_0^{2\pi} \frac{d}{d\theta} g_{\eps}^{[<n]}(\theta+ \omega )\cdot g_{\eps}^{[<n]}(\theta) d\theta \\
    &\;\;\; - \int_0^{2\pi} \frac{d}{d\theta} g_{\eps}^{[<n]}(\theta )\cdot g_{\eps}^{[<n]}(\theta - \omega) d\theta.
\end{align*}
Finally, since $g_0$ is a constant one has that $$ \eps^n \int_0^{2\pi} \frac{d}{d\theta} g_{\eps}^{[<n]}(\theta)\cdot R_n(\theta)d\theta \sim O(|\eps|^{n+1}). $$ 

Hence, putting all this together into the relation \eqref{eq:expr-R-zero-aver} we obtain that $\int_0^{2\pi}k\cdot R_n(\theta)d\theta = 0$.

To obtain the second relation in \eqref{eq:zero-R-n} recall that $R_n$ is the coefficient of order $\eps^{n-1}$ in the expansion in $\eps$ of $V'(\theta k +\sum_{i=0}^{n-1}g_i(\theta)\eps^i + \beta_{n-1} \eps^{n-1}k^\perp) $.
Since the term
$ \beta_{n-1} \eps^{n-1}k^\perp$ 
is of the highest order, we see that in the expansion of
$R_n$  it will appear linearly. 

Therefore,  we see that
\[
R_n(\theta) = \beta_{n-1} (n-1)! D^2V(k\theta + g_0)k^\perp + S_n(\theta)
\]
where $S_n(\theta)$ is an expression depending on $g_0, g_1(\theta),\ldots, g_{n-1}(\theta) $.  Thus, using the non-degeneracy assumption  \eqref{lowerassumption}, 
there exists a unique
$\beta_{n-1}\in \real$ such that $\int_0^{2\pi}k^\perp\cdot R_n(g_0, \ldots, g_{n-1}(\theta) +\beta k^\perp)d \theta = 0$.

\end{proof}
\begin{remark}    Note that the solutions, $g_n(\theta)$, of equations \eqref{eq:eq-ord-n}  are chosen in such a way that they satisfy the normalization \begin{equation}\label{eq:norm-low-gn}
         \int_0^{2 \pi} k\cdot g_n(\theta)d\theta = 0.
    \end{equation}
    Since the additive constant that is adjusted, according to Lemma \ref{lem:constants-lower-dim}, is proportional to $k^\perp$, the normalization \eqref{eq:norm-low-gn} is preserved. 
\end{remark}

\begin{remark}
We point out that due to \eqref{eq:existance-beta0} there are at least two possible choices for $\beta_0\in \real$ such that $g_0 =\beta_0 k^\perp $. Once that $\beta_0$ is chosen the formal solution $g_\eps = \sum \eps^n g_n$ , satisfying \eqref{eq:norm-low-gn}, will be unique.  
\end{remark}

\begin{remark} If the assumption \eqref{lowerassumption} fails, it is very
  easy to find examples of perturbations when the $\beta_n$ cannot be found
  and, therefore, there are no solutions in the sense of formal power series.
 \end{remark} 

\subsubsection{Lower dimensiona tori, Dissipative case ($\gamma \neq 0$)}\label{sec:sec:recur-low-dissip}
In this case the expansion $\mu_\eps = \sum \eps^n\mu_n$ also needs to be considered. To solve \eqref{invariance-lower} in the sense of formal power series the following sequence of equations have to be solved:
\begin{equation}\label{eq:g-whisk-12} 
    \LL_\omega g_n(\th) = R_n(\th) +\mu_n \quad\mbox{for } 0\leq n \leq 2 
\end{equation}
\begin{equation}\label{eq:g-whisk-3}
    \LL_\omega g_3(\th) = R_3(\th) -\gamma \omega k+ \mu_3
\end{equation}
\begin{equation}\label{eq:g-whisk-ord-n}
    \LL_\omega g_n(\th) = R_n(\th) + \mu_n  - \gamma g_{n-3}(\th) + \gamma g_{n-3}(\th-\omega) \quad\mbox{for } n \geq 4
\end{equation}

Equations \eqref{eq:g-whisk-12}, \eqref{eq:g-whisk-3},
\eqref{eq:g-whisk-ord-n} can be solved as follows: for the orders
$n=0,1,2$ one can choose $\mu_n=0$ and find $g_n$, solving
\eqref{eq:g-whisk-12} and satisfying \eqref{eq:norm-low-gn}, as in the
conservative case, Section \ref{sec:sec:recur-low-conserv}. When
$n=3$, one chooses $\mu_3 = \gamma \omega k$ and, again, solve
\eqref{eq:g-whisk-3} as in the conservative case. We point out that in
the cases $n=0, 1, 2, 3$ one can still use Lemma
\ref{lem:constants-lower-dim} to adjust an additive constant
proportional to $k^\perp$, due to the choices for $\mu_n$,
$n=0,1,2,3.$

We note that, as in the conservative case, there are at least two
possible choices for $\beta_0$ such that $g_0=\beta_0 k^\perp$. Once
$\beta_0\in\real $ is chosen, there is only one $\beta_{n-1}\in
\real$, $n\geq 2$, such that the equation at order $n$ will have a
solution and such that $g_{n-1} + \beta_{n-1}k^\perp$ solves the
cohomology equation of order $n-1$ and satisfies the normalization
\eqref{eq:norm-low-gn}.

For $n\geq 4$, one can proceed as follows: assume that for $i<n$ we
have found $g_i(\theta)$ solutions of \eqref{eq:g-whisk-12},
\eqref{eq:g-whisk-3}, \eqref{eq:g-whisk-ord-n} satisfying the normalization
\eqref{eq:norm-low-gn} and that, as in the conservative case,
$g_{n-1}(\theta)$ has been found up to an additive constant,
$\beta_{n-1}k^\perp$.

To find $g_n(\theta)$ we need that the right hand side of
\eqref{eq:g-whisk-ord-n} has zero average, that is,
$\int_0^{2\pi}\tilde R_n(\theta) d\theta = 0$ where $\tilde
R_n(\theta):= R_n(\theta)+\mu_n -\gamma g_{n-3}(\theta) +\gamma
g_{n-3}(\theta -\omega)$. Again, this is equivalent
to \begin{equation}\label{eq:RHS-zero-low-diss} \int_0^{2\pi} k\cdot
  \tilde R_n(\theta) d\theta = 0, \qquad \int_0^{2\pi} k^\perp \cdot
  \tilde R_n(\theta) d\theta = 0.
\end{equation} 
The first condition in \eqref{eq:RHS-zero-low-diss} is obtained by choosing \begin{equation}\label{eq:mu-low-dim}
\mu_n = - \left(\frac{1}{ (k\cdot k)}\int_0^{2\pi} k\cdot R_n(\theta) d\th \right) k.
\end{equation}

To fulfill the second condition in \eqref{eq:RHS-zero-low-diss} it is
enough to adjust an additive constant to $g_{n-1}(\theta)$ as in the
conservative case. To achieve this we need the same non-degeneracy
condition as in the conservative case, that is
\eqref{lowerassumption}. To fix the additive constant,
$\beta_{n-1}k^\perp$, note that $R_n(g_0, g_1(\theta), \ldots,
g_{n-1(\theta)}+\beta_{n-1} k^\perp)$ is the coefficient of order
$\eps^{n-1}$ in the expansion in $\eps$ of $V'(\theta k + \sum_{i=0}^{n-1}g_i(\theta)\eps^i + \beta_{n-1} \eps^{n-1}k^\perp).
$ Expanding we see that $$\tilde R_n(\theta) = \beta_{n-1} (n-1)!
D^2V(\theta k + g_0)k^\perp + S_n(\theta) + \mu_n -\gamma g_{n-3}(\theta)
+\gamma g_{n-3}(\theta -\omega),$$ where $S_n(\theta)$ is an
expression depending on $g_0, g_1(\theta),\ldots, g_{n-1}(\theta). $
Thus, if $$\int_0^{2\pi} k^\perp \cdot D^2 V(\theta k + g_0) k^\perp d\theta \neq 0 $$ there exists $\beta_{n-1}\in \real$ such that
$\int_0^{2\pi}k^\perp\cdot \tilde R_n(\theta)d \theta = 0$ .

\bigskip
In both of the previous cases, $\gamma = 0$ and $\gamma\neq 0$,  the computation of the coefficients $R_n(\th)$ in the right hand side of the equations \eqref{eq:eq-ord-n}, \eqref{eq:g-whisk-12}, \eqref{eq:g-whisk-3}, \eqref{eq:g-whisk-ord-n} is done in the exact same way, that is, \begin{equation}R_n(\th) =\sum_{|\ell|\leq J}{\alpha}_\ell F^{\ell, k}_{n-1}(\th)\label{eq:R-n-whisk}\end{equation} where the coefficients $F_n^{\ell,k}(\theta)$ satisfy the recursion  \eqref{eq:recursion-Froeshle} and $F_0^{\ell,k}(\theta)= e^{i\ell \cdot (\theta k + g_0) }$. We also have the following proposition.

\begin{proposition}
    \label{prop: Froeshle-trigonometric} 
  Assume that $V'(\th) =  \sum_{|\ell|\leq J} \alpha_\ell e^{ i \ell \cdot \th}$ is a trigonometric polynomial of degree $J$. Then, $g_n$ is trigonometric polynomial of degree less or equal than $nJ$ and $F_n^{\ell, k}$ is a trigonometric polynomial of degree less or equal than $(n+1)J$. Hence,  $\sum_{ |\ell|\leq J } \alpha_\ell F_n^{\ell,k}(\th)$ is a trigonometric polynomial of degree less or equal than $(n+1)J$.
  \end{proposition}

The proof of Proposition \ref{prop: Froeshle-trigonometric} is carried out exactly in the same way as the proof of Proposition \ref{trigonometric}.

\begin{remark}
    Note that the coefficient $R_n$ in \eqref{eq:R-n-whisk} are computed in the exact same way both in the conservative ($\gamma =0$) and the dissipative ($\gamma \neq 0)$ case. However, since $F^{\ell, k}_n(\th)$ depends on $g_1, \ldots, g_n$; the values of $R_n(\th)$ will be different whether one is solving \eqref{eq:eq-ord-n} or \eqref{eq:g-whisk-12}, \eqref{eq:g-whisk-3}, \eqref{eq:g-whisk-ord-n}.  
\end{remark}

\begin{remark}
    We point out that in the dissipative case, $\gamma \neq 0$, one can also prove the Gevrey character of the power series if one considers a dissipation of the form $1-\gamma\eps^m$, $m \in \nat$. In this case the cohomolgy equation at order $n > m$ is $$\LL_\omega g_n = R_n(\theta) + \mu_n - \gamma g_{n-m}(\theta) +\gamma g_{n-m}(\theta -\omega). $$ 
    The estimates given in Section \ref{sec:proofestimates} would also follow in a similar way. We have decided to work with the power $m=3$ to keep the proofs, and the quantifiers, as simple as possible. 
\end{remark}

\section{Proof of part B) of  Theorem~\ref{thm:main}}
\label{sec:proofestimates}

\subsection{Product properties of Gevrey series}
The first result is to show that the product of Gevrey formal power
series is also a Gevrey series of the same exponent.

\begin{proposition}\label{prop:product}
  Let $\| \cdot \|$ be a norm  which is  a Banach algebra under multiplication.

  If for $n \le N $, $N \in \nat  \cup \{\infty\}$ we have:
  
  \begin{equation} \label{assumedbounds}
    \begin{split}
    & \|u_n\| \le A (n!)^\sigma \\
    & \|v_n\| \le B (n!)^\sigma
  \end{split}
  \end{equation} 
  then, for all  $n \le N$ we have:
\begin{equation} \label{productformula} 
  \left\| \sum_{j = 0}^n  u_{n -j} v_j \right\|  \le \Gamma_\sigma A B (n!)^\sigma
\end{equation}
\end{proposition}

The formulation  of the property that we use in
the the proofs  is that if we have  that $u_n,v_n$ satisfy
\eqref{assumedbounds} for $n < N$, then the product satisfies the
bounds \eqref{productformula} for $n = N$. The key is that the
constant $\Gamma_\sigma$ is independent of $u_n,v_n$ and $N$.

\begin{remark}\label{gevreyalgebra} 
A consequence of of Proposition~\ref{prop:product}
is that if in the spaces of formal power series
we define the norm (take any $R, \sigma  > 0 $)
\[
  \| u \|_{R,\sigma} = \sup_{n \in \nat} \| u_n\| R^{n} (n!)^{-\sigma}
\]
we have  that the set of formal power series  with a finite $\|\cdot \|_{R,\sigma}$ are a Banach algebra under the multiplication of formal power series.
\begin{equation}\label{banach-algebra} 
  \| u \times v  \|_{R,\sigma}\leq C
\| u   \|_{R,\sigma}  \|v  \|_{R,\sigma}
\end{equation}

This clearly simplifies several arguments. For us, the finite version
is more practical since it allows to write more easily induction arguments
in the order of truncation.

\end{remark} 

An easy  corollary of the Banach algebra property \eqref{banach-algebra}
is
\begin{proposition} \label{analyticomposition} 
Assume  $\varphi(t) = \sum_n t^n \phi_n $ is an analytic
function and that for some $b> 0$ we have $\sum_n | \phi_n| C^{n-1} b^n 
 \equiv A < \infty$, where
 $C$ is the constant in \eqref{banach-algebra}.

 Then
 $\| \phi\circ u \|_{R, \sigma} \le A $, if $\|u \|_{R, \sigma} \le b$.
 \end{proposition}

 The proof of Proposition~\ref{analyticomposition} is
 just applying the triangle inequality to all the products and
 the Banach algebra inequality \eqref{banach-algebra}.
 Note also that the coefficients $(\phi\circ u ) $ are
 an expressions depending on $u_0, u_1,\ldots, u_n$.

\begin{proposition} \label{prop:technical}
  For any $\sigma > 0 $ there exists $\Gamma_\sigma> 0$ such
  that for all  $n \in \nat$
  \begin{equation} \label{productgevrey}
	  \sum_{j = 0}^n \binom{n}{j}^{-\sigma}  \le \Gamma_\sigma 
  \end{equation} 
\end{proposition}

The proof of Proposition~\ref{prop:product}
using Proposition~\ref{prop:technical} is very easy.

\begin{equation*}
  \begin{split} 
 \left \| \sum_{j = 0}^n  u_{n -j} v_j \right \|   
	  & \le \sum_{j = 0}^n \|  u_{n -j}\| \|  v_j \|  
 \le  \sum_{j = 0}^n  A ( (n-j)!)^{\sigma}  B (j!)^\sigma  \\
   &\le AB (n!)^\sigma \sum_{j=0}^n \left(\frac{j!(n-j)!}{n!}\right)^\sigma  \le (n!)^\sigma A B \Gamma_\sigma.
  \end{split}
\end{equation*}

\medskip 
Now we turn to the proof of Proposition~\ref{prop:technical}. The key insight is that, except for the edges $j \approx 0, n$;
    the binomial coefficient is approximated by a Gaussian.  This
    suggests that to prove \eqref{productgevrey}  we
    divide the sum into two pieces.

    Denote by $[ \cdot]$ the integer part, we have
    \[
      \begin{split}
      \sum_{j = 0}^n \binom{n}{j}^{-\sigma}  &\le
      \sum_{j = 0}^{[n/2]} \binom{n}{j}^{-\sigma}  +
      \sum_{j = n - [n/2]}^{n} \binom{n}{j}^{-\sigma}  \\
      &=  2     \sum_{j = 0}^{[n/2]} \binom{n}{j}^{-\sigma}
      \end{split}
    \]
    The first inequality is true becaue the sums in the  R.H.S. include
    all the terms in the sum in the L.H.S.  The inequality is strict
    with $n$ is even because then $\binom{n}{n/2}$ appears twice
    in the R.H.S.  The last equality follows because
    $\binom{n}{k} = \binom{n}{n-k}$. 
    
    We use the notation $f_n(k) := (n-k)! k!$. Then, we observe that
    $f_n(k) =  \frac{k}{n-k+1} f_n(k-1)$ and that, therefore,  $f_n(k)$ is decreasing
    for $k \le [n/2]$. Hence, if we fix any number $L$ independent of $n$
    (we will use later any $L$ such that $(L+1) \sigma > 1$), 
    we  have for all $n$ sufficiently large:
    \begin{equation}\label{zerobounds} 
      \begin{split} 
      \sum_{j = 0}^{[n/2]} \binom{n}{j}^{-\sigma}  &=
      \sum_{j = 0}^{[n/2]} f_n(j) ^{\sigma} (n!)^{-\sigma} \\
      & \le       \sum_{j = 0}^L f_n(j) ^{\sigma} (n!)^{-\sigma}
      + ([n/2] -L +1) f_n(L+1)^\sigma (n!)^{-\sigma} \\ 
      & =  1 + \left(\frac{1}{n}\right)   ^{\sigma}
      + \left(\frac{2!}{n(n-1)}\right)^{\sigma} + \cdots +
      \left(\frac{L!}{n(n-1) \cdots (n - L + 1) }\right)^{\sigma} \\
      &\phantom{AAAAAA}+ \left(\frac{(L+1)!}{n(n-1) \cdots (n - L) }\right)^{\sigma}( [n/2] - L +1 )
    \end{split}
  \end{equation} 
We see that if we take any $L$  the above estimate is
a finite number of explicit functions of $n$. The first terms in \eqref{zerobounds} 
are clearly going to zero  as $n$ tends to infinity.

The  last term in \eqref{zerobounds} has in the denominator in the 
parenthesis
$L+1$ factors that are asymptotic to $n$ so that  the last term in \eqref{zerobounds} 
 is asymptotic to $n^{-(L+1)\sigma}  [n/2] $. 
Hence if $L$ is large enough that 
\begin{equation} \label{Lbounds} 
(L+1) \sigma > 1
\end{equation}  this last term also tends to zero as 
$n$ goes to infinity. 

Hence, we obtain that, under \eqref{Lbounds},  the sum
in \eqref{zerobounds}  is bounded uniformly in $n$. 

\subsection{Estimates on the recursions and end of
    the proof of Part B) of Theorem~\ref{thm:main}}

  We  fix one of the norms \eqref{norms}, that is, we consider $\rho$ and $r$ fixed.  The proof remains valid for all
  of them even if the region of validity of the constants
  may be altered. Since the constants can be adjusted by scaling the
  $\eps$, as we will show in Section~\ref{subsec:scaling}, a very explicit choice
  of the constant does
  not affect the results.

  We present estimates for the recursive procedure to construct the
  formal power series solutions specified in Section~\ref{sec:partA}. We first show that, if we have the desired conclusion for
  a range of all orders up to a certain order, $n-1$, we can get the  same desired conclusion for the next order, $n$. This induction argument (Lemma~\ref{lem:inductivefinal}, Lemma \ref{lem:inductive-froeshle-dissip}, Lemma \ref{lem:inductive-froeshle})
  depends on some inequalities among the constants involved. A second step is to show that, by scaling the original variable,
  we can arrange the inequalities assumed
  in Lemmas~\ref{lem:inductivefinal}, \ref{lem:inductive-froeshle-dissip}, \ref{lem:inductive-froeshle} as well as ensuring the first
  steps of the induction.

\smallskip

 We introduce the notation: 
 \begin{equation}\label{upsilon}\Upsilon = \sum_{|\ell|\leq J} |\alpha_\ell| \end{equation}
 where $\alpha_\ell$ are the coefficients in \eqref{polynomial}. 
 Then, $\Upsilon$ is a  measure of the size of the nonlinear
 perturbations.  A crucial assumption in our results is
 that $\Upsilon$ is small enough. This can be arranged by scaling $\eps$.
 
 \subsubsection{Maximal dimensional tori}
 \begin{lemma}\label{lem:inductivefinal}
   We are in the set up of Section~\ref{sec:recursive}
   and we consider the formal power series $u = \sum \eps^n u_n, \mu =\sum \eps^n \mu_n$ solving
   \eqref{invariance2}.

   Assume:
   
   The frequency  $\om\in \real^D$ is a  Diophantine multiple of $2 \pi$, of type $(\nu, \tau)$
   (See Definition~\ref{diophantine}). For $0 \le |\ell| \le  J$, and for $n > n_0 \ge 4$,  we have 
  \begin{equation} \label{eq:induction} 
 \begin{split}
     & \| E^\ell_j \| \le A  (j!)^\sigma \quad  0 \leq  j \le n-1 \\
     & \| u_j \| \le B (j!)^\sigma \quad 0 \le j \leq n-1
   \end{split}
   \end{equation}
   Then, provided that the constants in \eqref{eq:induction} satisfy:
   \begin{equation} \label{condition1} 
     \begin{split}
       & 2 \tau <  \sigma \\
       & J\Gamma_\sigma AB  \leq A \\
       & 4\nu^{-2} J^{2\tau}\left( \Upsilon A  + 2\gamma B \right)
       \leq B  \\
     \end{split}
   \end{equation}
   we recover the same assumptions
   \eqref{eq:induction} with $n$ in place of $n-1$.
\end{lemma} 

\begin{proof}
  Again, we note that we have freely used rather rough estimates
  (e.g. we have not used that the first terms in the recursion  are
  different). The goal is to get a simple proof. For example, we have decided to estimate only the terms of
  the expansion for $n \geq n_0$ to avoid dealing with the fact that
  the terms of order smaller than $3$ have different formulas
  (because the dissipation is $1-\gamma \eps^3$). 

  Note that $u_n$ satisfies equation \eqref{eq:cohom-eq-order-n}, then  using Lemma \ref{lem:inverse}, Proposition \ref{trigonometric}, and  the notation $T_\omega(\th) := \th + \omega  $, we have

\begin{align*}
     \|u_n\|_{\rho,r} &\leq \left\|\LL_\omega^{-1}\left( \sum_{|\ell|\leq J } \alpha_\ell E^\ell_{n-1} +\mu_n - \gamma u_{n-3} + \gamma u_{n-3}\circ T_{ - \omega} \right)\right\|_{\rho,r} \\
     &\leq 4\nu^{-2} (nJ)^{2\tau} \left \|\sum_{|\ell |\leq J } \alpha_\ell E^\ell_{n-1} + \mu_n - \gamma u_{n-3} + \gamma u_{n-3}\circ T_{-\omega}  \right \|_{\rho,r} \\
     & \leq 4\nu^{-2}J^{2\tau}n^{2\tau} \left( \left\| \sum_{|\ell |\leq J } \alpha_\ell E^\ell_{n-1} \right\|_{\rho, r} + 2 \gamma \left\| u_{n-3} \right\|_{\rho, r} \right) \\
     & \leq 4\nu^{-2} J^{2\tau}n^{2\tau} \left(  \Upsilon  A((n-1)!)^\sigma + 2\gamma B( (n-3)!)^\sigma  \right)  \\
     & \leq  4\nu^{-2} J^{2\tau}\left( \Upsilon  A  + 2\gamma B \right) (n!)^\sigma 
\end{align*}

Where we have used Remark~\ref{projection} to estimate
\[ \left\| \sum_{|\ell |\leq J } \alpha_\ell E^\ell_{n-1}(\th) +\mu_n  \right\|_{\rho,r}
  \le \left\| \sum_{|\ell |\leq J } \alpha_\ell E^\ell_{n-1}(\th) \right\|_{\rho,r}
  \]

  Note also that we have used  very elementary estimates on
  the terms involving $u_{n-3}$.

Also, by \ref{recursion}
\begin{align*}
    \|E_n^\ell\|_{\rho,r} &\leq \frac{|\ell |}{n}\sum_{m=0}^{n-1}(m+1)\|u_{m+1}\|_{\rho,r} \|E^k_{n-1-m}\|_{\rho,r} \\
    & \leq \frac{|\ell |}{n}\sum_{m=0}^{n-1}(m+1) B ((m+1)!)^\sigma A ((n-1-m)!)^\sigma \\
    &\leq \frac{J}{n} AB  (n!)^\sigma \sum_{m=0}^{n-1} (m+1)\left( \frac{(m+1)!(n-1-m)!}{n!}\right)^\sigma \\
    & =  J(n!)^\sigma  AB \sum_{m=0}^{n-1} \frac{m+1}{n} {n \choose m+1}^{-\sigma} \\
    & \leq  J(n!)^\sigma  AB \sum_{m=1}^{n} {n \choose m}^{-\sigma} \\
    & \leq J\Gamma_\sigma AB  (n!)^\sigma  \\
\end{align*}
where $\Gamma_\sigma$ is the constant given by Proposition \ref{prop:technical}.
\end{proof}

\subsubsection{Properties of the quantities under scaling and end of
  the argument for maximal dimensional tori} \label{subsec:scaling}

Now we turn to arguing that by rescaling the perturbative parameter
$\eps$, we can adjust the conditions in \eqref{condition1}
and get the first terms of the induction started.

The main observation is that if we consider expansions in
a new variable $\teps$ such that $\eps = \teps \eta$ with
$\eta > 0$ a fixed number, which we will take to be small, we obtain a problem of the same form  but the parameters of the problem change, that is, \def\tUpsilon{ {\tilde \Upsilon}}
\def\tgamma{ {\tilde \gamma}}
$\tUpsilon = \Upsilon \eta$,  $\tgamma = \eta^3 \gamma$.

Furthermore, we observe that in the perturbation expansions
we have that $u_0 = 0 $ and, substituting
$\eta \teps$ in place of $\eps$ in the series, we obtain
that the scaled series satisfy: 
for $j \ge 1$, 
$\tilde u_j = \eta^j u_j $, Similarly, the trigonometric
expansions satisfy $E^\ell_0(\th)  = e^{i \ell\cdot \th} $, and, for
$j \ge 1$,  $\tilde E^\ell_j (\th)  = \eta^j E^\ell_j$ .

In summary, by scaling, we can assume that the constants $\Upsilon, \gamma$ are
small enough,  that the first terms $u_n$ are also small, and
that the $E^\ell_j$ are small for $1 \leq j \le n_0$ ($E^\ell_0$ remains of
a fixed size for all the scalings, that is,  $\| E^\ell_0 \|_{\rho, r} = \exp( |\ell| \rho)(1 + |\ell|^2)^{r/2} $ ).

Taking into account the observations above, now we show  that  conditions \eqref{condition1} are satisfied by choosing suitably the constants $\Upsilon$, $\gamma$ (this condition amount two smallness conditions in the scaling parameter $\eta$). We recall we are in a situation where $\rho$ and $r$, the parameters of the norm, are fixed. We start by choosing $\sigma > 2\tau $, then we choose $A$ such that $A > \exp(J\rho)(1+|J|^2)^{r/2}$ and $B$ small enough so that the second inequality in \eqref{condition1} is satisfied. We observe that if $\Upsilon$ and $\gamma$ are small enough, the third inequality in \eqref{condition1} is satisfied. As indicated above, this condition amounts to smallness conditions in the scaling parameter $\eta$. 

To end the argument we observe that by the choice of $A$ above, the scaling properties, and  by choosing the scaling parameter small enough we can ensure the inequalities in \eqref{eq:induction} for $0\leq j\leq 4$. Hence, under a finite number of conditions in $\eta$, we can have
the conditions in \eqref{condition1} and \eqref{eq:induction} fulfilled, i.e.,  the induction
starts and can be continued. Finally, in the original scaling we obtain that
$\| u_j\|_{\rho, r} \le  A \eta^{-j}( j!)^\sigma$, which is
the claimed result.

\begin{remark}
    The Gevrey estimates for $\mu_n$ come directly from \eqref{eq:mu-max-tori}, that is, $$\mu_n = -\int_{\torus^D}  \sum_{|\ell|\leq J} \alpha_\ell E_{n-1}^\ell(\th)\, d\th.$$ Once one has the estimates \eqref{eq:induction} the estimates for $\mu_n$ follow immediately. We have not included these estimates in the Lemmas to keep the statements as simple as possible.  
\end{remark}

\subsubsection{Lower dimensional tori. Dissipative case, $\gamma\neq 0$.}

\begin{lemma}\label{lem:inductive-froeshle-dissip}
   We are in the set up of Section~\ref{sec:sec:recur-low-dissip}. We consider the formal power series $g=\sum \eps^n g_n, \mu = \sum \eps^n\mu_n$ solving \eqref{invariance-lower}.
   Assume:
   
   The frequency  $\om\in \real$ is a  Diophantine multiple of $2 \pi$, of type $(\nu, \tau)$.

   Assume  also that for $0 \le |\ell| \le  J$,  we have 
  \begin{equation} \label{eq:induc-whisk-diss} 
 \begin{split}
     & \| F^{\ell, k}_j \| \le A  (j!)^\sigma \quad  0 \leq  j \le n-1 \\
     & \| g_j \| \le B (j!)^\sigma \quad 0 \le j \leq n-1
   \end{split}
   \end{equation}
   Then, provided that the constants in \eqref{eq:induc-whisk-diss} satisfy:
   \begin{equation} \label{condition-whisk-diss} 
     \begin{split}
       & 2 \tau <  \sigma \\
       & J\Gamma_\sigma AB  \leq A \\
       &  4\nu^{-2} J^{2\tau}\left( \Upsilon  A  + 2\gamma B \right)
       \leq B \\
     \end{split}
   \end{equation}
   we recover the same assumptions
   \eqref{eq:induc-whisk-diss} with $n$ in place of $n-1$.
\end{lemma} 
The proof of Lemma \ref{lem:inductive-froeshle-dissip} follows using the same arguments as in the proof of Lemma \ref{lem:inductivefinal}. First, the inequality for $\|g_i\|$ in \eqref{eq:induc-whisk-diss} is obtained in the exact same way as in Lemma \ref{lem:inductivefinal} because the equations  
\begin{equation*}
    \LL_\omega g_n(\th) = R_n(\th) + \mu_n  - \gamma g_{n-3}(\th) + \gamma g_{n-3}(\th-\omega) \quad\mbox{for } n \geq 4
\end{equation*}
in Section \eqref{sec:sec:recur-low-dissip}, have the same form as the ones in Section \ref{sec:recursive}. Recall that \begin{equation}\label{eq:Rn-low-dim-bis} R_n(\th) =\sum_{|\ell|\leq J}{\alpha}_\ell F^{\ell, k}_{n-1}(\th) .\end{equation}

Next, the inequality for $\|F_j^{\ell, k}\|$ in \eqref{eq:induc-whisk-diss} is also obtained in the exact same way as in Lemma \ref{lem:inductivefinal} due to the fact that  the recursion \begin{equation}
    n F_n^{\ell, k}(\th) = \sum_{m=0}^{n-1}(m+1) i\ell \cdot g_{m+1}(\th) F_{n-1-m}^{\ell, k}(\th).
\end{equation}
 has the same form as \eqref{recursion}. 

Finally, one can use the same scaling argument in Section \ref{subsec:scaling} to ensure the conditions \eqref{condition-whisk-diss}. 

\begin{remark}
    Again, the Gevrey estimates for the coefficients $\mu_n$ come from \eqref{eq:mu-low-dim}, that is,  \begin{equation}
\mu_n = - \left(\frac{1}{ (k\cdot k)}\int_0^{2\pi} k\cdot R_n(\theta) d\th \right) k 
\end{equation} and the fact that $R_n(\theta)$ is given by \eqref{eq:Rn-low-dim-bis}
\end{remark}

\subsubsection{Lower dimensional tori. Conservative case, $\gamma= 0$.}

\begin{lemma}\label{lem:inductive-froeshle}
   We are in the set up of Section~\ref{sec:sec:recur-low-conserv}.  We consider the formal power series $g=\sum \eps^n g_n$ solving \eqref{invariance-lower} with $\gamma =0$. Assume:
   
   The frequency  $\om\in \real$ is a Diophantine multiple of $2 \pi$, of type $(\nu,\tau)$.  For $0 \le |\ell| \le  J$,  we have 
  \begin{equation} \label{eq:induction-whiskered} 
 \begin{split}
     & \| F^{\ell, k}_j \| \le A  (j!)^\sigma \quad  0 \leq  j \le n-1 \\
     & \| g_j \| \le B (j!)^\sigma \quad 0 \le j \leq n-1
   \end{split}
   \end{equation}
   Then, provided that the constants in \eqref{eq:induction-whiskered} satisfy:
   \begin{equation} \label{condition1-whiskered} 
     \begin{split}
       & 2 \tau <  \sigma \\
       & J\Gamma_\sigma AB  \leq A \\
       & 4 \nu^{-2}  J^{2\tau}\Upsilon A 
       \leq B  \\
     \end{split}
   \end{equation}
   we recover the same assumptions
   \eqref{eq:induction-whiskered} with $n$ in place of $n-1$.
\end{lemma} 

\begin{proof}
Using \eqref{eq:eq-ord-n}, Lemma \ref{lem:inverse}, Proposition \ref{prop: Froeshle-trigonometric}, and \eqref{eq:R-n-whisk} we have

\begin{align*}
     \|g_n\|_{\rho,r} &\leq \left\|\LL_\omega^{-1}\left( \sum_{|\ell|\leq J } \alpha_\ell F^{\ell,k}_{n-1}(\th) \right) \right\|_{\rho,r} \\
     &\leq 4\nu^{-2} (nJ)^{2\tau} \|\sum_{|\ell|\leq J } \alpha_\ell F^{\ell, k}_{n-1}\|_{\rho,r}  \\
     & \leq 4\nu^{-2} J^{2\tau}n^{2\tau}A\Upsilon  ((n-1)!)^\sigma  \\
                      & \leq 4 \nu^{-2}  J^{2\tau}\Upsilon A 
                        (n!)^\sigma 
\end{align*}

Also, using the recursion \eqref{eq:recursion-Froeshle}
\begin{align*}
    \|F_n^{\ell, k}\|_{\rho,r} &\leq \frac{|\ell|}{n}\sum_{m=0}^{n-1}(m+1)\|g_{m+1}\|_{\rho,r} \|F_{n-1-m}^{\ell, k}\|_{\rho,r} \\
    & \leq \frac{|\ell|}{n}\sum_{m=0}^{n-1}(m+1) B ((m+1)!)^\sigma A ((n-1-m)!)^\sigma \\
    &\leq \frac{J}{n} AB  (n!)^\sigma \sum_{m=0}^{n-1} (m+1)\left( \frac{(m+1)!(n-1-m)!}{n!}\right)^\sigma \\
    & =  J(n!)^\sigma  AB \sum_{m=0}^{n-1} \frac{m+1}{n} {n \choose m+1}^{-\sigma} \\
    & \leq  J(n!)^\sigma  AB \sum_{m=1}^{n} {n \choose m}^{-\sigma} \\
    & \leq J\Gamma_\sigma AB  (n!)^\sigma  \\
\end{align*}
where $\Gamma_\sigma$ is the constant given in Proposition \eqref{prop:technical}. 

Finally, we note that the same argument in Section \ref{subsec:scaling} can be applied for lower dimensional tori. That is, considering a scaling $\tilde\eps = \eta \eps$ one obtains the scaled constant $\tilde \Upsilon =\eta \Upsilon $, and choosing $\eta$ small enough one can assure that \eqref{condition1-whiskered} is fulfilled.
\end{proof}

\bibliographystyle{alpha}
\bibliography{gevrey.bib}

\begin{thebibliography}{CCDlL13b}

\bibitem[BC19]{BustamanteC19}
Adri{\'a}n~P. Bustamante and Renato~C. Calleja.
\newblock Computation of domains of analyticity for the dissipative standard
  map in the limit of small dissipation.
\newblock {\em Physica D}, 395:15--23, 2019.

\bibitem[BC21]{BustamanteC21}
Adri{\'a}n~P. Bustamante and Renato~C. Calleja.
\newblock Corrigendum and addendum to: ``{Computation} of domains of
  analyticity for the dissipative standard map in the limit of small
  dissipation''.
\newblock {\em Physica D}, 417:7, 2021.
\newblock Id/No 132837.

\bibitem[BDlL22]{BustamanteL22}
Adri{\'a}n~P. Bustamante and Rafael De~la Llave.
\newblock Gevrey estimates for asymptotic expansions of tori in weakly
  dissipative systems.
\newblock {\em Nonlinearity}, 35(5):2424--2473, 2022.

\bibitem[BK78]{BrentK78}
R.~P. Brent and H.~T. Kung.
\newblock Fast algorithms for manipulating formal power series.
\newblock {\em J. Assoc. Comput. Mach.}, 25(4):581--595, 1978.

\bibitem[Car63]{Cartan63}
H.~Cartan.
\newblock Elementary theory of analytic functions of one or several complex
  variables. {Translated} from the {French} by {John} {Standring} and
  {H}.{B}.{Shutrick}.
\newblock Adiwes {Interantional} {Series}. {Reading}, {Mass}.-{Palo}
  {Alto}-{London}: {Addison}- {Wesley} {Publishing} {Company}, {Inc}. 226 p.
  (1963)., 1963.

\bibitem[CCdlL13a]{CallejaCL13}
Renato~C. Calleja, Alessandra Celletti, and Rafael de~la Llave.
\newblock A {KAM} theory for conformally symplectic systems: efficient
  algorithms and their validation.
\newblock {\em J. Differential Equations}, 255(5):978--1049, 2013.

\bibitem[CCDlL13b]{Ca-Ce-Lla-13}
Renato~C Calleja, Alessandra Celletti, and Rafael De~la Llave.
\newblock A {KAM} theory for conformally symplectic systems: efficient
  algorithms and their validation.
\newblock {\em Journal of Differential Equations}, 255(5):978--1049, 2013.

\bibitem[CCdlL17]{CallejaCL17}
Renato~C. Calleja, Alessandra Celletti, and Rafael de~la Llave.
\newblock Domains of analyticity and {L}indstedt expansions of {KAM} tori in
  some dissipative perturbations of {H}amiltonian systems.
\newblock {\em Nonlinearity}, 30(8):3151--3202, 2017.

\bibitem[CCdlL20]{CallejaCL20}
Renato~C. Calleja, Alessandra Celletti, and Rafael de~la Llave.
\newblock Existence of whiskered {KAM} tori of conformally symplectic systems.
\newblock {\em Nonlinearity}, 33(1):538--597, 2020.

\bibitem[CCdlL22]{CallejaCL22}
Renato~C. Calleja, Alessandra Celletti, and Rafael de~la Llave.
\newblock {KAM} quasi-periodic solutions for the dissipative standard map.
\newblock {\em Commun. Nonlinear Sci. Numer. Simul.}, 106:29, 2022.
\newblock Id/No 106111.

\bibitem[CG08]{CorsiG08}
Livia Corsi and Guido Gentile.
\newblock Melnikov theory to all orders and {P}uiseux series for subharmonic
  solutions.
\newblock {\em J. Math. Phys.}, 49(11):112701, 29, 2008.

\bibitem[Chi79]{Chirikov}
Boris~V. Chirikov.
\newblock A universal instability of many-dimensional oscillator systems.
\newblock {\em Phys. Rep.}, 52(5):264--379, 1979.

\bibitem[dlL01]{Llave01}
R.~de~la Llave.
\newblock A tutorial on {KAM} theory.
\newblock In {\em Smooth ergodic theory and its applications ({S}eattle, {WA},
  1999)}, volume~69 of {\em Proc. Sympos. Pure Math.}, pages 175--292. Amer.
  Math. Soc., Providence, RI, 2001.

\bibitem[dlLGJV05]{LlaveGJV05}
R.~de~la Llave, A.~Gonz\'{a}lez, \`A. Jorba, and J.~Villanueva.
\newblock K{AM} theory without action-angle variables.
\newblock {\em Nonlinearity}, 18(2):855--895, 2005.

\bibitem[FHM]{FernandezHM22}
\'Alvaro Fern\'andez, Alex Haro, and J.~M. Mondelo.
\newblock Flow map parameterization methods for invariant tori in
  quasi-periodic {H}amiltonian systems.

\bibitem[Fro71]{Froeschle71}
Claude Froeschl\'e.
\newblock On the number of isolating integrals in systems with three degrees of
  freedom.
\newblock {\em Astrophys. Space Sci.}, 14:110--117, 1971.

\bibitem[GG02]{GallavottiG02}
G.~Gallavotti and G.~Gentile.
\newblock Hyperbolic low-dimensional invariant tori and summations of divergent
  series.
\newblock {\em Comm. Math. Phys.}, 227(3):421--460, 2002.

\bibitem[JLZ99]{Jo-Lla-Zou-99}
Angel Jorba, Rafael De~La Llave, and Maorong Zou.
\newblock Lindstedt series for lower dimensional tori.
\newblock In {\em Hamiltonian systems with three or more degrees of freedom},
  pages 151--167. Springer, 1999.

\bibitem[KAdlL22]{KumarAL22}
Bhanu Kumar, Rodney~L. Anderson, and Rafael de~la Llave.
\newblock Rapid and accurate methods for computing whiskered tori and their
  manifolds in periodically perturbed planar circular restricted 3-body
  problems.
\newblock {\em Celestial Mech. Dynam. Astronom.}, 134(1):Paper No. 3, 38, 2022.

\bibitem[Knu98]{Knuth}
Donald~E. Knuth.
\newblock {\em The art of computer programming. {V}ol. 2}.
\newblock Addison-Wesley, Reading, MA, 1998.
\newblock Seminumerical algorithms, Third edition [of MR0286318].

\bibitem[Mos67]{Moser67}
J\"{u}rgen Moser.
\newblock Convergent series expansions for quasi-periodic motions.
\newblock {\em Math. Ann.}, 169:136--176, 1967.

\bibitem[Poi05]{PoincareS}
H.~Poincar{\'e}.
\newblock Le{\c{c}}ons de {M{\'e}canique} c{\'e}leste profess{\'e}es {\`a} la
  {Sorbonne}. {Tome} {I}. {Th{\'e}orie} g{\'e}n{\'e}rale des perturbations
  plan{\'e}taires.
\newblock Paris: {Gauthier}-{Villars}. {VI} u. 367 {S}. {{\(8^{\circ}\)}}
  (1905)., 1905.

\bibitem[Poi87]{Poincare}
H.~Poincar\'{e}.
\newblock {\em Les m\'{e}thodes nouvelles de la m\'{e}canique c\'{e}leste.
  {T}ome {II}}.
\newblock Les Grands Classiques Gauthier-Villars. [Gauthier-Villars Great
  Classics]. Librairie Scientifique et Technique Albert Blanchard, Paris, 1987.
\newblock M\'{e}thodes de MM. Newcomb, Gyld\'{e}n, Lindstedt et Bohlin. [The
  methods of Newcomb, Gyld\'{e}n, Lindstedt and Bohlin], Reprint of the 1893
  original, Biblioth\`eque Scientifique Albert Blanchard. [Albert Blanchard
  Scientific Library].

\bibitem[Tre89]{Treshev}
D.~V. Treshch\"{e}v.
\newblock A mechanism for the destruction of resonance tori in {H}amiltonian
  systems.
\newblock {\em Mat. Sb.}, 180(10):1325--1346, 1439, 1989.

\bibitem[XdlLW22]{XuLW22}
Xiaodan Xu, Rafael de~la Llave, and Fenfen Wang.
\newblock The existence of solutions for nonlinear elliptic equations: simple
  proofs and extensions of a paper by {Y}. {Shi}.
\newblock {\em J. Differ. Equations}, 318:20--57, 2022.

\end{thebibliography}

\end{document}